\documentclass[10pt]{amsart}
\usepackage{amsmath}%
\usepackage{amsfonts}%
\usepackage{amssymb}%
\usepackage{graphicx, color, enumerate}
\usepackage{latexsym}
\usepackage{amscd}
\usepackage{mathrsfs}
\usepackage{cancel}
\usepackage{mathtools}
\usepackage[top=1in, bottom=1.25in, left=1.25in, right=1.25in]{geometry}
\usepackage{tikz}
 \usepackage{tikz-cd}
 \usetikzlibrary{matrix,arrows,decorations.pathmorphing}
 \usetikzlibrary[patterns]
\newtheorem{theorem}{Theorem}[section]
\theoremstyle{plain}
\newtheorem*{theorem*}{Theorem}
\theoremstyle{plain}

\newtheorem{corollary}{Corollary}[section]
\newtheorem{definition}{Definition}[section]
\newtheorem{lemma}{Lemma}[section]

\newtheorem{remark}{Remark}[section]

\def\dys{\displaystyle}

\numberwithin{equation}{section}
\DeclareRobustCommand{\rchi}{{\mathpalette\irchi\relax}}
\newcommand{\irchi}[2]{\raisebox{\depth}{$#1\chi$}}
\begin{document}
\title[]{Regularity of solutions to a fractional elliptic problem with mixed Dirichlet-Neumann boundary data}

\author{J. Carmona}
\address[J. Carmona]{ Departamento de Matem\'aticas,
Universidad de Almer\'ia,
\newline Ctra. Sacramento s/n, La Ca\~nada de San Urbano, 04120 Almer\'ia,   Spain
}
\email{jcarmona@ual.es}

\author{E. Colorado}
\address[E. Colorado, A. Ortega]{Departamento de Matem\'aticas,
Universidad Carlos III de Madrid, \newline Av. Universidad 30, 28911 Legan\'es (Madrid), Spain}
\email[E. Colorado]{ecolorad@math.uc3m.es}%
\email[A. Ortega]{alortega@math.uc3m.es}

\author{T. Leonori}
\address[T. Leonori]{Dipartimento di Scienze di Base e Applicate per l'Ingegneria
Universit\`a di Roma \lq\lq Sapienza\rq\rq \newline
Via Antonio Scarpa 10,
00161 Roma}
\email{tommaso.leonori@sbai.uniroma1.it}%

\author{A. Ortega}

\date{\today}
\thanks{E. Colorado and A. Ortega are partially supported
by the Ministry of Economy and Competitiveness of \ \ Spain and FEDER
under Research Project MTM2016-80618-P. J. Carmona is partially supported by Ministerio de Econom\'ia y Competitividad (MINECO-FEDER), Spain under grant MTM2015-68210-P and Junta de Andaluc\'{\i}a FQM-194.}

\begin{abstract}
In this work we study regularity properties of solutions to
fractional elliptic problems with mixed Dirichlet-Neumann boundary
data when dealing with the Spectral Fractional Laplacian.
\end{abstract}
\maketitle
%\tableofcontents

\section{Introduction}

In this paper we study some regularity properties of the solutions
to fractional elliptic problems such as
\begin{equation}\label{problema1}
        \left\{
        \begin{tabular}{rcl}
        $(-\Delta)^su=f$ & &\mbox{in }$\Omega$, \\
        $B(u)=0$  & &on $\partial\Omega$, \\
        \end{tabular}
        \right.
                \tag{$P^s$}
\end{equation}
where $\frac{1}{2}<s<1$, $f\in L^p(\Omega),\ p>\frac{N}{2s}$ and
$\Omega$ is a bounded  domain of $\mathbb{R}^N$, $N\geq1$. By
$B(u)$ we mean the mixed Dirichlet-Neumann boundary condition, i.e.
\begin{equation*}
B(u)=u\rchi_{\Sigma_{\mathcal{D}}}+\frac{\partial u}{\partial \nu} \rchi_{\Sigma_{\mathcal{N}}},
\end{equation*}
where $\nu$ is the outwards normal to $\partial\Omega$, $\rchi_A$ stands for the characteristic function of the set $A$ and $\Omega$ satisfy
\begin{equation*}
        (\mathfrak{B})\ \left\{
        {\renewcommand{\arraystretch}{1.2}\begin{tabular}{c}
        $\Omega\subset \mathbb{R}^N$ is  a bounded Lipschitz domain \\
       $\Sigma_{\mathcal{D}}$ and $\Sigma_{\mathcal{N}}$ are smooth $(N-1)$-dimensional submanifolds of $\partial\Omega$, \\
       $\Sigma_{\mathcal{D}}$ is a closed manifold of positive $(N-1)$-dimensional Lebesgue measure,\\
                 $\displaystyle  | \Sigma_{\mathcal{D}}|=\alpha\in(0,|\partial\Omega|)$.\\
             $\Sigma_{\mathcal{D}}\cap\Sigma_{\mathcal{N}}=\emptyset\,,\ \Sigma_{\mathcal{D}}\cup\Sigma_{\mathcal{N}}=\partial\Omega\mbox{  and  }\Sigma_{\mathcal{D}}\cap\overline{\Sigma}_{\mathcal{N}}=\Gamma\,$ \\
              $\Gamma$ is a smooth
            $(N-2)$-dimensional submanifold of $\partial\Omega$.
        \end{tabular}}
        \right.
\end{equation*}
The main result we prove here is the following.
\begin{theorem}\label{holder_result}
Assume that $\Omega$  satisfies   hypotheses
$(\mathfrak{B})$ and let $u$ be the solution to problem \eqref{problema1} with $\frac{1}{2}<s<1$, $f\in L^p(\Omega),\ p>\frac{N}{2s}$.
Then $u\in \mathcal{C}^\gamma(\overline{\Omega})$ for some $0<\gamma<\frac{1}{2}$. Moreover, there exists a constant
$\mathscr{H}=\mathscr{H}(N,s,f,p,|\Sigma_{\mathcal{D}}|)>0$ such that
\begin{equation*}
|u(x)-u(y)|\leq\mathscr{H}|x-y|^{\gamma},\ \forall\ x,y\in\overline{\Omega}.
\end{equation*}
\end{theorem}

To prove Theorem \ref{holder_result} we follow some of the ideas in \cite{KS,S}. Using the De Giorgi truncation method, Stampacchia (see \cite{S})
established the regularity of solutions to the mixed boundary problem involving the classical Laplace operator. Due to the nonlocal nature of problem
\eqref{problema1}, some difficulties arise when trying to apply this truncation method to solutions to \eqref{problema1}. Based on the ideas of
\cite{CS,CT,BrCdPS}, at this point we will make full use of the local realization of the fractional operator $(-\Delta)^s$ in terms of certain
auxiliary degenerate elliptic problem. We use the results of \cite{FKS} to adapt the procedures of \cite{S} to the case of degenerate elliptic
equations with weights in the Muckenhoupt class $A_2$ (see \cite{FKS} for the precise definition as well as some useful properties of those weights).
\newline
In addition to Theorem \ref{holder_result}, following some ideas in \cite{ColP}, in the last part of the work we study the behaviour of the problem \eqref{problema1}
when we move the boundary condition in a regular way as follows. Given $I_{\varepsilon}=[\varepsilon,|\partial\Omega|]$ for some $\varepsilon>0$,
let us consider the family of closed sets $\{\Sigma_{\mathcal{D}}(\alpha)\}_{\alpha\in I_{\varepsilon}}$, satisfying
\begin{itemize}
\item[$(B_1)$] $\Sigma_{\mathcal{D}}(\alpha)$ has a finite number of connected components.
\item[$(B_2)$] $\Sigma_{\mathcal{D}}(\alpha_1)\subset\Sigma_{\mathcal{D}}(\alpha_2)$ if $\alpha_1<\alpha_2$.
\item[$(B_3)$]  $|\Sigma_{\mathcal{D}}(\alpha_1)|=\alpha_1\in I_{\varepsilon}$.
\end{itemize}
We denote by $\Sigma_{\mathcal{N}}(\alpha)=\partial\Omega\backslash\Sigma_{\mathcal{D}}(\alpha)$ and $\Gamma(\alpha)=\Sigma_{\mathcal{D}}(\alpha)\cap\overline{\Sigma}_{\mathcal{N}}(\alpha)$. For a family of this type we consider the corresponding family of mixed boundary value problems
\begin{equation} \label{p_move}
        \left\{
        \begin{tabular}{rcl}
        $(-\Delta)^su=f$ & &in $\Omega\subset \mathbb{R}^{n}$, \\
        $B_{\alpha}(u)=0$  & &on $\partial\Omega$, \\
        \end{tabular}
        \right.
        \tag{$P_{\alpha}^s$}
\end{equation}
where $B_{\alpha}(u)$ is the boundary condition associated to the parameter $\alpha$ in the previous hypotheses and the boundary manifolds $\Sigma_{\mathcal{D}}(\alpha)$ and $\Sigma_{\mathcal{N}}(\alpha)$ satisfy the corresponding hypotheses $(\mathfrak{B}_{\alpha})$. In this scenario we prove the following result.
\begin{theorem}\label{cor:reg}
Given $\Omega$ a smooth bounded domain such that the family $\{\Sigma_{\mathcal{D}}(\alpha)\}_{\alpha\in I_{\varepsilon}}$
satisfies the hypotheses $(\mathfrak{B}_{\alpha})$ and $(B_1)$--$(B_3)$, let $u_{\alpha}$ be the solution to
\eqref{p_move} with $\frac{1}{2}<s<1$, $f\in L^p(\Omega)$ and $\ p>\frac{N}{2s}$. Then, there exist two constants $0<\gamma<\frac12$ and $\mathscr{H}_{\varepsilon}>0$ both independent from $\alpha\in[\varepsilon,|\partial\Omega|]$ such that
\begin{equation*}
\|u_{\alpha}\|_{\mathcal{C}^{\gamma} (\overline{\Omega})}\leq \mathscr{H}_{\varepsilon}.
\end{equation*}
\end{theorem}
As we will see in the proof of Theorem \ref{cor:reg}, when one takes $\alpha\to0^+$ the control of the H\"older norm of such a family is lost. Hence, it is necessary bound from below  the measure of the family $\{\Sigma_{\mathcal{D}}(\alpha)\}_{\alpha\in I_{\varepsilon}}$, in order to guarantee the control on the H\"older norm for the family $\{u_{\alpha}\}_{\alpha\in I_{\varepsilon}}$.

\medskip 

Let us stress that problem related to the spectral fractional Laplacian with mixed boundary conditions are news and, to our knowledge, have been treated only in \cite{CCLO, CO}.

\section{Functional setting and preliminaries}
As far as the fractional Laplace operator is concerned, we recall its definition given through the spectral decomposition.
Let $(\varphi_i,\lambda_i)$ be the
eigenfunctions (normalized with respect to the $L^2(\Omega)$-norm)
and the eigenvalues of $(-\Delta)$ equipped with homogeneous mixed Dirichlet-Neumann boundary data.
Then, $(\varphi_i,\lambda_i^s)$ are the eigenfunctions and
eigenvalues of the fractional operator $(-\Delta)^s$, where given $\displaystyle u_i(x)=\sum_{j\geq1}\langle u_i,\varphi_j\rangle\varphi_j$, $i=1,2$
\[
\langle
 (-\Delta)^s u_1, u_2
\rangle
=
\sum_{j\ge 1} \lambda_j^s\langle u_1,\varphi_j\rangle \langle u_2,\varphi_j\rangle,
\]
i.e., the action of the fractional operator on a smooth function $u_1$ is given by
\begin{equation*}
(-\Delta)^su_1=\sum_{j\ge 1} \lambda_j^s\langle u_1,\varphi_j\rangle\varphi_j.
\end{equation*}
As a consequence, the fractional Laplace operator $(-\Delta)^s$ is well defined through its spectral decomposition in the
following space of functions that vanish on $\Sigma_{\mathcal{D}}$,
\begin{equation*}
H_{\Sigma_{\mathcal{D}}}^s(\Omega)=\left\{u=\sum_{j\ge 1} a_j\varphi_j\in L^2(\Omega):\ \|u\|_{H_{\Sigma_{\mathcal{D}}^s}(\Omega)}^2=
\sum_{j\ge 1} a_j^2\lambda_j^s<\infty\right\}.
\end{equation*}
Observe that since $u\in H_{\Sigma_{\mathcal{D}}}^s(\Omega)$, it follows that
\begin{equation*}
\|u\|_{H_{\Sigma_{\mathcal{D}}}^s(\Omega)}=\left\|(-\Delta)^{\frac{s}{2}}u\right\|_{L^2(\Omega)}.
\end{equation*}
As it is proved in \cite[Theorem 11.1]{LM}, if $0<s\le \frac{1}{2}$ then $H_0^s(\Omega)=H^s(\Omega)$ and, therefore, also $H_{\Sigma_{\mathcal{D}}}^s(\Omega)=H^s(\Omega)$, while for $\frac 12<s<1$, $H_0^s(\Omega)\subsetneq H^s(\Omega)$. Hence, the range $\frac 12<s<1$ guarantees that $H_{\Sigma_{\mathcal{D}}}^s(\Omega)\subsetneq H^s(\Omega)$, provides us the correct functional space to study the mixed boundary problem \eqref{problema1}.\newline
This definition of the fractional powers of the Laplace operator allows us to integrate by parts in the appropriate spaces, so that a natural definition of weak solution to problem $(P_{s})$ is the following.
\begin{definition}
We say that $u\in H_{\Sigma_{\mathcal{D}}}^s(\Omega)$ is a solution to \eqref{problema1} if
\begin{equation*}
\int_{\Omega}(-\Delta)^{s/2}u \,(-\Delta)^{s/2}\psi dx=\int_{\Omega}f\psi dx,\ \ \text{for any}\ \psi\in H_{\Sigma_{\mathcal{D}}}^s(\Omega).
\end{equation*}
\end{definition}
Due to the nonlocal nature of the fractional operator $(-\Delta)^s$ some difficulties arise when one tries to obtain an explicit expression  of the action of the fractional Laplacian on a given function. In order to overcome this difficuly, we use the ideas by Caffarelli and Silvestre (see \cite{CS})  together with those of \cite{BrCdPS, CT}   to give an equivalent definition of the operator $(-\Delta)^s$ by means of an auxiliary problem that we introduce next. 

Given any domain $\Omega \subset \mathbb{R}^N$, we set the cylinder $\mathscr{C}_{\Omega}=\Omega\times(0,\infty)\subset\mathbb{R}_+^{N+1}$. We denote by $(x,y)$ those points that belong to $\mathscr{C}_{\Omega}$ and by $\partial_L\mathscr{C}_{\Omega}=\partial\Omega\times[0,\infty)$ the lateral boundary of the  cylinder. Let us also denote by  $\Sigma_{\mathcal{D}}^*=\Sigma_{\mathcal{D}}\times[0,\infty)$ and $\Sigma_{\mathcal{N}}^*=\Sigma_{\mathcal{N}}\times[0,\infty)$ as well as $\Gamma^*=\Gamma\times[0,\infty)$. It is clear that, by construction,
\begin{equation*}
\Sigma_{\mathcal{D}}^*\cap\Sigma_{\mathcal{N}}^*=\emptyset\,, \quad \Sigma_{\mathcal{D}}^*\cup\Sigma_{\mathcal{N}}^*=\partial_L\mathscr{C}_{\Omega} \quad \mbox{and} \quad \Sigma_{\mathcal{D}}^*\cap\overline{\Sigma_{\mathcal{N}}^*}=\Gamma^*\,.
\end{equation*}
 Given a function $u\in H_{\Sigma_{\mathcal{D}}}^s(\Omega)$ we define its $s$-harmonic extension function, denoted by $U (x,y)=E_{s}[u(x)]$, as the solution to the problem
\begin{equation*}
        \left\{
        \begin{array}{rlcl}
        \displaystyle   -\text{div}(y^{1-2s}\nabla U (x,y))&\!\!\!\!=0  & & \mbox{ in } \mathscr{C}_{\Omega} , \\
        \displaystyle B(U(x,y) )&\!\!\!\!=0   & & \mbox{ on } \partial_L\mathscr{C}_{\Omega} , \\
         \displaystyle U(x,0)&\!\!\!\!=u(x)  & &  \mbox{ on } \Omega\times\{y=0\} .
        \end{array}
        \right.
\end{equation*}
where
\begin{equation*}
B(U)=U\rchi_{\Sigma_{\mathcal{D}}^*}+\frac{\partial U}{\partial \nu}\rchi_{\Sigma_{\mathcal{N}}^* },
\end{equation*}
being $\nu$, with an abuse of notation\footnote{Let $\nu$ be the outwards normal to $\partial\Omega$ and $\nu_{(x,y)}$ the outwards normal
 to $\mathscr{C}_{\Omega}$ then, by construction, $\nu_{(x,y)}=(\nu,0)$, $y>0$.}, the outwards normal to
 $\partial_L\mathscr{C}_{\Omega}$.
 Following the well known result by Caffarelli and Silvestre (see \cite{CS}), $U$ is related to the fractional Laplacian of
 the original function through the formula
\begin{equation*}
\frac{\partial U}{\partial \nu^s}:= -\kappa_s \lim_{y\to 0^+} y^{1-2s}\frac{\partial U}{\partial y}=(-\Delta)^su(x),
\end{equation*}
where $\kappa_s$ is a  suitable positive constant (see \cite{BrCdPS} for its exact value). The extension function belongs to the space
\begin{equation*}
\mathcal{X}_{\Sigma_{\mathcal{D}}}^s(\mathscr{C}_{\Omega}) : =\overline{\mathcal{C}_{0}^{\infty}
((\Omega\cup\Sigma_{\mathcal{N}})\times[0,\infty))}^{\|\cdot\|_{\mathcal{X}_{\Sigma_{\mathcal{D}}}^s(\mathscr{C}_{\Omega})}},
\end{equation*}
where we define
$$
\|\cdot\|_{\mathcal{X}_{\Sigma_{\mathcal{D}}}^s(\mathscr{C}_{\Omega})}^2:=\kappa_s\int_{\mathscr{C}_{\Omega}}\mkern-5mu y^{1-2s} |\nabla (\cdot)|^2dxdy.
$$
Note that $\mathcal{X}_{\Sigma_{\mathcal{D}}}^s(\mathscr{C}_{\Omega})$ is a Hilbert space equipped with the norm
$\|\cdot\|_{\mathcal{X}_{\Sigma_{\mathcal{D}}}^s(\mathscr{C}_{\Omega})}$ which is induced by the scalar product
\begin{equation*}
\langle U, V \rangle_{\mathcal{X}_{\Sigma_{\mathcal{D}}}^s(\mathscr{C}_{\Omega})}=\kappa_s
\int_{\mathscr{C}_{\Omega}}y^{1-2s} \langle\nabla U,\nabla V\rangle dxdy.
\end{equation*}
Moreover, the following inclusions are satisfied,
\begin{equation} \label{embedd}
\mathcal{X}_0^s(\mathscr{C}_{\Omega}) \subset \mathcal{X}_{\Sigma_{\mathcal{D}}}^s(\mathscr{C}_{\Omega}) \subsetneq \mathcal{X}^s(\mathscr{C}_{\Omega}),
\end{equation}
being  $\mathcal{X}_0^s(\mathscr{C}_{\Omega})$ the space of functions that belongs to $\mathcal{X}^s(\mathscr{C}_{\Omega})\equiv H^1(\mathscr{C}_{\Omega},y^{1-2s}dxdy)$ and vanish on the lateral boundary of $\mathscr{C}_{\Omega}$.\newline
Using the above arguments we can reformulate the problem \eqref{problema1} in terms of the extension problem as follows:
\begin{equation}\label{extension_problem}
        \left\{
        \begin{array}{rlcl}
        \displaystyle   -\text{div}(y^{1-2s}\nabla U)&\!\!\!\!=0  & & \mbox{ in } \mathscr{C}_{\Omega} , \\
        \displaystyle B(U)&\!\!\!\!=0   & & \mbox{ on } \partial_L\mathscr{C}_{\Omega} , \\
         \displaystyle \frac{\partial U}{\partial \nu^s}&\!\!\!\!=f  & &  \mbox{ on } \Omega\times\{y=0\} .
        \end{array}
        \right.
        \tag{$P_{s}^*$}
\end{equation}

Next, we specify

\begin{definition}
An {\rm energy solution} to problem \eqref{extension_problem} is a function $U\in \mathcal{X}_{\Sigma_{\mathcal{D}}}^s(\mathscr{C}_{\Omega})$ such that
\begin{equation}\label{def:soldebil}
\kappa_s\int_{\mathscr{C}_{\Omega}} y^{1-2s} \langle\nabla U,\nabla\varphi\rangle  \  dxdy=\int_{\Omega} f(x)\varphi(x,0)dx, \qquad \forall  \varphi\in \mathcal{X}_{\Sigma_{\mathcal{D}}}^s(\mathscr{C}_{\Omega}).
\end{equation}
\end{definition}

If  $U\in \mathcal{X}_{\Sigma_{\mathcal{D}}}^s(\mathscr{C}_{\Omega})$ is the solution to problem \eqref{extension_problem} we can associate the
function $u (x) =Tr[U(x,y) ]=U(x,0)$, that  belongs to $H_{\Sigma_{\mathcal{D}}}^s(\Omega)$, and solves problem \eqref{problema1}.
Moreover, also the vice versa is true: given a solution $u \in H_{\Sigma_{\mathcal{D}}}^s(\Omega)$ we can define its $s$-harmonic extension
$U\in \mathcal{X}_{\Sigma_{\mathcal{D}}}^s(\mathscr{C}_{\Omega})$,
as the solution to \eqref{extension_problem}.
Thus, both formulations are equivalent and the {\it Extension operator}
$$
E_s: H_{\Sigma_{\mathcal{D}}}^s(\Omega) \to \mathcal{X}_{\Sigma_{\mathcal{D}}}^s(\mathscr{C}_{\Omega}),
$$
allows us to switch between both of them.

Accordingly to \cite{CS,BrCdPS}, due to the choice of the constant $\kappa_s$, the extension operator $E_s$ is an isometry, i.e.
\begin{equation}\label{norma2}
\|E_s[\varphi] (x,y) \|_{\mathcal{X}_{\Sigma_{\mathcal{D}}}^s(\mathscr{C}_{\Omega})}=
\|\varphi (x) \|_{H_{\Sigma_{\mathcal{D}}}^s(\Omega)},\ \text{for all}\ \varphi\in H_{\Sigma_{\mathcal{D}}}^s(\Omega).
\end{equation}

Let us also recall the \textit{trace inequality}, that is   a  useful tool we exploit in many proofs in this paper (see \cite{BrCdPS}):

there exists  $  C=C(N,s,r,|\Omega|)$
such that $\forall z\in\mathcal{X}_0^s(\mathscr{C}_{\Omega})$
\begin{equation*}
C
\left(\int_{\Omega}|z(x,0)|^rdx\right)^{\frac{2}{r}}
\leq \int_{\mathscr{C}_{\Omega}}y^{1-2s}|\nabla z(x,y)|^2dxdy ,
\end{equation*}
with  $1\leq r\leq 2^*_s,\ N>2s$, with $2^*_s= \frac{2N}{N-2s}$.

Observe that such inequality turns out to be, in fact,
equivalent to the fractional Sobolev inequality:
\begin{equation*}
  C\left(\int_{\Omega}|v|^rdx\right)^{\frac{2}{r}}
\leq \int_{\Omega}|(-\Delta)^{\frac{s}2}v|^2dx
,\qquad
\forall v\in H_{0}^s(\Omega),\ 1\leq r\leq 2^*_s ,\ N>2s.
\end{equation*}
When mixed boundary conditions are considered, the situation is quite similar since the Dirichlet condition is imposed on a set
$\Sigma_{\mathcal{D}} \subset \partial \Omega$ such that $|\Sigma_{\mathcal{D}}|=\alpha>0$. Hence, thanks to \eqref{embedd},
there exists a positive constant $C_{\mathcal{D}}=C_{\mathcal{D}}(N,s,|\Sigma_{\mathcal{D}}|)$ such that
\begin{equation}\label{const}
0<\inf_{\substack{u\in H_{\Sigma_{\mathcal{D}}}^s(\Omega)\\ u\not\equiv 0}}\frac{\|u\|_{H_{\Sigma_{\mathcal{D}}}^s(\Omega)}^2}{
\|u\|_{L^{2_s^*}(\Omega)}^2}:=C_{\mathcal{D}}
<\inf_{\substack{u\in H_{0}^s(\Omega)\\ u\not\equiv 0}}\frac{\|u\|_{H_{0}^s(\Omega)}^2}{\|u\|_{L^{2_s^*}(\Omega)}^2}.
\end{equation}
\begin{remark}%\label{rem:sobconst}
It is worth to observe (see \cite{CO}, \cite{CCLO}) that
$C_{\mathcal{D}}(N,s,|\Sigma_{\mathcal{D}}|)\leq
2^{-\frac{2s}{N}}C(N,s,2_s^*)$. Moreover, having in mind the
spectral definition of the fractional operator and by H\"older
inequality, it follows that $C_{\mathcal{D}}\leq
|\Omega|^{\frac{2s}{N}}\lambda_1^s(\alpha)$, with
$\lambda_1(\alpha)$ the first eigenvalue of the Laplace operator
with mixed boundary conditions on the sets
$\Sigma_{\mathcal{D}}=\Sigma_{\mathcal{D}}(\alpha)$ and
$\Sigma_{\mathcal{N}}= \Sigma_{\mathcal{N}}(\alpha)$. Since
$\lambda_1(\alpha)\to0$ as $\alpha\to0^+$, see \cite[Lemma
4.3]{ColP}, we conclude that $C_{\mathcal{D}}\to0$ as
$\alpha\to0^+$.
\end{remark}

Gathering together \eqref{norma2} and \eqref{const}, we obtain,
\begin{equation}\label{poinc}
C_{\mathcal{D}}\left(\int_\Omega |\varphi(x,0)|^{2^*_s} dx\right)^{\frac{2}{2^*_s}}\leq
\|\varphi(x,0)\|_{H_{\Sigma_{\mathcal{D}}}^s(\Omega)}^2=
\|E_s[\varphi(x,0)]\|_{\mathcal{X}_{\Sigma_{\mathcal{D}}}^s(\mathscr{C}_{\Omega})}^2.
\end{equation}
With this Sobolev-type inequality in hand we can prove a trace inequality adapted to the mixed boundary data framework.

\begin{lemma}\label{lem:traceineq}
There exists a constant $C_{\mathcal{D}}=C_{\mathcal{D}}(N,s,|\Sigma_{\mathcal{D}}|)>0$ such that,
\begin{equation}\label{eq:traceineq}
C_{\mathcal{D}}
\left(\int_\Omega
 |\varphi(x,0)|^{2^*_s}  )dx\right)^{\frac{2}{2^*_s}}
\leq
\int_{\mathscr{C}_{\Omega}} y^{1-2s} |\nabla \varphi|^2 dxdy
,\qquad \forall\varphi \in \mathcal{X}_{\Sigma_{\mathcal{D}}}^s(\mathscr{C}_{\Omega}).
\end{equation}
\end{lemma}
\begin{proof}
Thanks to \eqref{poinc}, it is enough to prove that $\displaystyle\|E_s[\varphi(\cdot,0)]\|_{\mathcal{X}_{\Sigma_{\mathcal{D}}}^s(\mathscr{C}_{\Omega})}\leq\|\varphi \|_{\mathcal{X}_{\Sigma_{\mathcal{D}}}^s(\mathscr{C}_{\Omega})}$. This inequality is satisfied since, arguing as in \cite{BrCdPS}, we find
\begin{align*}
\|\varphi\|_{\mathcal{X}_{\Sigma_{\mathcal{D}}}^s(\mathscr{C}_{\Omega})}^2&:=\kappa_s\int_{\mathscr{C}_{\Omega}}\mkern-5mu y^{1-2s} |\nabla \varphi|^2dxdy\\
&=\kappa_s\int_{\mathscr{C}_{\Omega}}y^{1-2s} |\nabla \left(E_s[\varphi(x,0)]+\varphi (x,y) -E_s[\varphi(x,0)]\right)|^2dxdy\\
&=\|E_s[\varphi(x,0)]\|^2_{\mathcal{X}_{\Sigma_{\mathcal{D}}}^s(\mathscr{C}_{\Omega})}+ \|\varphi (x,y) -E(\varphi(x,0))\|^2_{\mathcal{X}_{\Sigma_{\mathcal{D}}}^s(\mathscr{C}_{\Omega})}\\
&\ +2\kappa_s\int_{\mathscr{C}_{\Omega}}y^{1-2s} \langle\nabla E_s[\varphi(x,0)], \nabla(\varphi (x,y) -E_s[\varphi(x,0)])\rangle dxdy\\
&=\|E_s[\varphi(x,0)]\|^2_{\mathcal{X}_{\Sigma_{\mathcal{D}}}^s(\mathscr{C}_{\Omega})}+
\|\varphi (x,y) -E_s[\varphi(x,0)]\|^2_{\mathcal{X}_{\Sigma_{\mathcal{D}}}^s(\mathscr{C}_{\Omega})}
\\
&\ +2\int_{\Omega}(-\Delta)^s (\varphi(x,0))(\varphi(x,0)-\varphi(x,0))dx
\\[1.5 ex]
&=\|E_s[\varphi(x,0)]\|^2_{\mathcal{X}_{\Sigma_{\mathcal{D}}}^s(\mathscr{C}_{\Omega})}+ \|\varphi (x,y) -E_s[\varphi(x,0)]\|^2_{\mathcal{X}_{\Sigma_{\mathcal{D}}}^s(\mathscr{C}_{\Omega})}.
\end{align*}
\end{proof}

\section{H\"older Regularity}
The principal result we prove in this Section is Theorem \ref{holder_result}, which deals with the H\"older regularity of the solution
to problem \eqref{problema1}. First we introduce the notation that we will follow along this Section.\newline
{\bf Notation.}
Given an open bounded set $\Omega$, $x\in\overline{\Omega}\subset\mathbb{R}^{N}$ and
$X\in\overline{\mathscr{C}}_{\Omega}\subset\mathbb{R}_+^{N+1}$, we define
\begin{itemize}
\item[--] $\Omega(x,\rho)=\Omega\cap B_{\rho}(x)$,
\item[--] $\mathscr{C}_{\Omega}(X,\rho)=\mathscr{C}_{\Omega}\cap B_{\rho}(X)$,
\end{itemize}
Given $u(x) \in H_{\Sigma_{\mathcal{D}}}^s(\Omega)$ and $U(X) \in \mathcal{X}_{\Sigma_{\mathcal{D}}}^s(\mathscr{C}_{\Omega})$,
let us also define
\begin{itemize}
\item[--] $A_+(k)=\{x\in\Omega: u(x)>k\}$,
\item[--] $A_+^*(k)=\{X\in\mathscr{C}_{\Omega}: U(X)>k\}$,
\item[--] $A_+(k,\rho)=A_+(k)\cap\Omega(x,\rho)$
\item[--] $A_+^*(k,\rho)=A_+^*(k)\cap\mathscr{C}_{\Omega}(X,\rho)$,
\item[--] $\{\cdot\}^k=\min(\cdot,k)$.
\item[--] $\{\cdot\}_k=\max(\cdot,k)$.
\end{itemize}
In a similar way we may define the sets $A_{-}(k)$, $A_{-}^{*}(k)$, $A_{-}(k,\rho)$ and $A_{-}^{*}(k,\rho)$ replacing $>$ with $<$ in the latter definitions. We denote by
\begin{itemize}
\item[--] $|A|_{\omega}$ the measure induced by a weight $\omega$ of the set $A$.
\item[--] $|A|_{y^{1-2s}}$ the measure induced by the weight $y^{1-2s}$ of the set $A$.
\item[--] $|A|$ the usual Lebesgue measure of the set $A$.
\end{itemize}

\subsection*{On the regularity of $\Omega$}

Let us recall that $\Omega $ is assumed,  in all the paper, to be Lipschitz and consequently also $\mathcal{C}_{\Omega}$ turns out to have the same regularity.
In particular, among others, we use the following properties.
There exists $\zeta\in (0,1)$ such that
for any $z \in \overline{\Omega}$ and any $\rho>0$
\begin{equation}\label{ineq:zeta}
|\mathscr{C}_{\Omega}(Z,\rho)| \geq\zeta|B_{\rho}(Z)|.
\end{equation}
Moreover also the weighted counterpart is true, i.e.
there exists $\zeta_s\in (0,1)$ such that
for any $z \in \overline{\Omega}$ and any $\rho>0$
\begin{equation}\label{ineq:zetaw}
|\mathscr{C}_{\Omega}(Z,\rho)|_{y^{1-2s}} \geq\zeta_s|B_{\rho}(Z)|_{y^{1-2s}}.
\end{equation}
Consequently $\exists \lambda>0 $ such that
\begin{equation}\label{hyp:thcota}
|A_+^*(k,r)|_{y^{1-2s}}\leq\lambda|\mathscr{C}_{\Omega(z,R)}(Z,r)|_{y^{1-2s}}.
\end{equation}

\bigskip

It is worth to observe that all the results we prove in this paper might be proved for a larger class of open sets $\Omega$. Indeed following \cite{S}, this kind of results is true for the so called $\frac12$--admissible domains. Here we decided to not deal with such domains for brevity and in order to not  make the proofs much heavier.

\bigskip

Now we are ready to start with the statement and the proofs of several technical results.

\bigskip

Let $z\in\overline{\Omega}$ and $R>0$ and let $u$ be  a solution  to problem \eqref{problema1}: we write $u(x)=v(x)+w(x)$ for every $x\in \Omega(z,R)$, where the function $v(x)$ satisfies
\begin{equation}\label{nonhomo_abajo}
\begin{cases}
(-\Delta)^sv=f  &  \mbox{ in  } \Omega(z,R),  \\
\mkern+50mu v=0            &  \mbox{ on } \widetilde{\Sigma}_{\mathcal{D},R}:=\partial\Omega(z,R)\backslash{\Sigma}_{\mathcal{N}}, \\
\mkern+35mu\displaystyle \frac{\partial v }{ \partial \nu} =0 & \mbox{ on } \widetilde{\Sigma}_{\mathcal{N},R}:=\partial\Omega(z,R)\cap\Sigma_{\mathcal{N}},
\end{cases}
\end{equation}
and the function $w(x)$ is such that,
\begin{equation}\label{homo_abajo}
\begin{cases}
(-\Delta)^sw=0  &  \mbox{ in  } \Omega(z,R),  \\
\mkern+50mu w=0  &  \mbox{ on  }  \Sigma_{\mathcal{D},R}:=\Sigma_{\mathcal{D}}\cap B_R(z),  \\
\mkern+36mu\displaystyle \frac{\partial w}{\partial \nu}=0
&  \mbox{ on  }   \Sigma_{\mathcal{N},R}:=\Sigma_{\mathcal{N}}\cap B_R(z),
\end{cases}
\end{equation}
Using the extension technique we can write $v(x)=V(x,0)$ with $V(x,y)$ solves the extended problem
\begin{equation}\label{nonhomo_arriba}
\begin{cases}
     -\text{div }(y^{1-2s}\nabla V)=0  & \mbox{  in } \mathscr{C}_{\Omega(z,R)}, \\
\mkern+88mu B(V)=0   & \mbox{ on } \partial_L\mathscr{C}_{\Omega(z,R)}, \\
\mkern+99mu\displaystyle \frac{\partial V}{\partial \nu^s}=f  & \mbox{  on } \Omega(z,R)\times\{y=0\},
        \end{cases}
\end{equation}
where $\displaystyle B(V)=V\rchi_{\widetilde{\Sigma}_{\mathcal{D},R}^*}\mkern-10mu+\frac{\partial V}{\partial \nu}\rchi_{\widetilde{\Sigma}_{\mathcal{D},R}^*}$,
with $\widetilde{\Sigma}_{\mathcal{D},R}^*\!=\!\widetilde{\Sigma}_{\mathcal{D},R}\!\times\![0,\infty)$ and $\widetilde{\Sigma}_{\mathcal{N},R}^*=\widetilde{\Sigma}_{\mathcal{N},R}\!\times\![0,\infty)$.

In the same way, we write $w(x)=W(x,0)$, with $W(x,y)$ satisfying the extended problem
\begin{equation}\label{homo}
\begin{cases}
-\text{div}(y^{1-2s}\nabla W)=0  &  \mbox{ in  } \mathscr{C}_{\Omega(z,R)},\\
\mkern+81mu B(W)=0& \mbox{ on } \Sigma_{\mathcal{D},R}^*\cup\Sigma_{\mathcal{N},R}^*,\\
\mkern+95mu\displaystyle  \frac{\partial W}{\partial\nu^s}=0  &  \mbox{  on  } \Omega(z,R)\times\{y=0\},
\end{cases}
\end{equation}
where $\displaystyle B(V)=V\rchi_{\Sigma_{\mathcal{D},R}^*}\mkern-10mu+\frac{\partial V}{\partial \nu}\rchi_{\Sigma_{\mathcal{D},R}^*}$,
with $\Sigma_{\mathcal{D},R}^*\!=\!{\Sigma}_{\mathcal{D},R}\!\times\![0,\infty)$ and $\Sigma_{\mathcal{N},R}^*={\Sigma}_{\mathcal{N},R}\!\times\![0,\infty)$.

\medskip

Let us observe that we have the following situations:

\begin{enumerate}[(i)]
\item If $z\in\Omega$, there exists $R>0$ such that $\widetilde{\Sigma}_{\mathcal{D},R}=\partial\Omega(z,R)$ and $\Sigma_{\mathcal{D},R}=\Sigma_{\mathcal{N},R}=\emptyset$. Then, $v\in H_0^s(\Omega(z,R))$ and it is solution to a Dirichlet problem. Moreover,  $w$ is an $s$--harmonic function,  i.e. its extension  $W=E_s[w] \in\mathcal{X}^s(\mathscr{C}_{\Omega(z,R)}) $   and it satisfies
   \begin{equation}\label{soldebil}
\int_{\mathscr{C}_{\Omega(z,R)}} y^{1-2s} \langle\nabla W,\nabla\Phi\rangle  \  dxdy=0, \qquad \forall  \Phi\in \mathcal{X}_0^s(\mathscr{C}_{\Omega(z,R)}).
\end{equation}
\item If $z\in\Sigma_{\mathcal{D}}\backslash\Gamma$, there exists $R>0$ such that $\widetilde{\Sigma}_{\mathcal{D},R}=\partial\Omega(z,R)$ and $\Sigma_{\mathcal{N},R}=\emptyset$, then, $v\in H_0^s(\Omega(z,R))$ and it is a solution to a Dirichlet problem while  $W\in \mathcal{X}_{\Sigma_{\mathcal{D},R}}^s(\mathscr{C}_{\Omega(z,R)})$ and, also in this case, it satisfies \eqref{soldebil}.
% holds $\forall\Phi\in \mathcal{X}_0^s(\mathscr{C}_{\Omega(z,R)})$.

\item If $z\in\Sigma_{\mathcal{N}}$, there exists $R>0$ such that $\Sigma_{\mathcal{D},R}=\emptyset$. Then, the function $v\in H_{\widetilde{\Sigma}_{\mathcal{D},R}}^s(\Omega(z,R))$ and it is a solution to the mixed problem \eqref{nonhomo_abajo}; moreover    $W$ belongs to $\mathcal{X}^s(\mathscr{C}_{\Omega(z,R)})$ and \eqref{soldebil} holds $\forall\Phi\in \mathcal{X}^s(\mathscr{C}_{\Omega(z,R)})$ vanishing on $\partial_L\mathscr{C}_{\Omega(z,R)}\backslash\Sigma_{\mathcal{N},R}^*$.

\item Finally, if $z\in\Gamma$, the sets $\widetilde{\Sigma}_{\mathcal{D},R}$, $\widetilde{\Sigma}_{\mathcal{N},R}$, $\Sigma_{\mathcal{D},R}$ and $\Sigma_{\mathcal{D},R}$ are nonempty for all $R>0$. Then, the function $v\in H_{\widetilde{\Sigma}_{\mathcal{D},R}}^s(\Omega(z,R))$ and it is a solution to the mixed problem \eqref{nonhomo_abajo}; as far as $w$ is concerned,  $W\in \mathcal{X}_{\Sigma_{\mathcal{D},R}}^s(\mathscr{C}_{\Omega(z,R)})$ and if fulfills  \eqref{soldebil} holds for any $\Phi\in \mathcal{X}^s(\mathscr{C}_{\Omega(z,R)})$ vanishing on $\partial_L\mathscr{C}_{\Omega(z,R)}\backslash\Sigma_{\mathcal{N},R}^*$.
\end{enumerate}

We also define the following sets that will be useful in the sequel:
\begin{itemize}
\item $\mathscr{C}_{\Omega(z,R)}^\circ=\overline{\mathscr{C}}_{\Omega(z,R)}\backslash \{ (x,y)\in \mathscr{C}_{\Omega(z,R)}: x\in\partial B_R(z)\},$
\item $\partial_0\mathscr{C}_{\Omega(z,R)}=\partial_L\mathscr{C}_{\Omega(z,R)}\backslash\Sigma_{\mathcal{N},R}^*$.
\item $\partial_{B}\mathscr{C}_{\Omega(z,R)}=\partial_L\mathscr{C}_{\Omega(z,R)}\backslash\left(\Sigma_{\mathcal{D},R}^*\cup\Sigma_{\mathcal{N},R}^*\right)$.
\end{itemize}

We continue by stating the definitions and results needed in what follows. The first definition is based on \cite[Definition 2.1]{S}.

\begin{definition}\label{def:truncate}
Given any $z_0 \in \overline{\Omega}$ and  $Z\in\!\mathscr{C}_{\Omega(z_0,R)}^\circ$, let $\mathcal{K}^+(Z)\,($resp.\,$\mathcal{K}^-(Z))$ be the set of values $k\!\in\!\mathbb{R}$  such that there exists a number $\widetilde{\rho}(Z)>0$ satisfying $\{U\}^k\eta\in \mathcal{X}_{\partial_0\mathscr{C}_{\Omega(z_0,R)}}^s(\mathscr{C}_{\Omega(z_0,R)})$ $($resp. $\{U\}_k\,\eta\in \mathcal{X}_{\partial_0\mathscr{C}_{\Omega(z_0,R)}}^s(\mathscr{C}_{\Omega(z_0,R)}))$ for any $U\in \mathcal{X}_{\Sigma_{\mathcal{D},R}}^s(\mathscr{C}_{\Omega(z,R)})$ and any function $\eta\in C^{\infty}(\mathbb{R}_+^{N+1})$ such that $supp(\eta)\subset B_{\widetilde{\rho}(Z)}(Z)$.
\end{definition}
\begin{remark}
It is worth to  observe that:
\begin{itemize}
\item[--] If $Z\in\! \Sigma_{\mathcal{D},R}^*$
%\setminus \partial_{B}\mathscr{C}_{\Omega(z,R)}$,
then $\mathcal{K}^+(Z)\!=\![0,\infty)$, $\mathcal{K}^-(Z)\!=\!(-\infty,0]$ and $\widetilde{\rho}(Z)=dist(Z,\partial_{B}\mathscr{C}_{\Omega(z,R)})$.
\item[--] If $Z \in \mathscr{C}_{\Omega(z,R)}^\circ \backslash\Sigma_{\mathcal{D},R}^*$, then $\mathcal{K}^+(Z)=\mathcal{K}^-(Z)=(-\infty,\infty)$, and in this case \break $\widetilde{\rho}(Z)=dist(Z,\partial_0\mathscr{C}_{\Omega(z,R)})$.
\item[--] Thanks to the construction of the   cylinder, it is immediate to notice that the number $\widetilde{\rho}(Z)>0$ does not depend on the $y$ variable.
\end{itemize}
\end{remark}

The control of the oscillations of solutions of elliptic problems is usually carried out through integral estimates that mainly rely on a Sobolev-type inequality.
Since the extension function solves a degenerate elliptic problem involving a weight (namely, $y^{1-2s}$) that belongs to the Muckenhoupt class $A_2$, it is necessary to establish a Sobolev-type inequality dealing with such a type of singular weights. To this aim, we recall the following definition.

\begin{definition}%\label{muck}
Given an open subset $D\subset \mathbb{R}^N$ and a   function $\omega: D \to \mathbb{R}^+$, we say that  $\omega$  belongs to the Muckenhoupt class $   A_p$, with $p>1$ if there exists a constant $C>0$ such that
$$
\sup_{B\subset D} \left( \frac1{|B|} \int_B \omega^p\right) \left( \frac1{|B|} \int_B \omega^{1-p}\right)^{p-1} \leq C\,.
$$
\end{definition}

Now we can recall the following result.

\begin{theorem}[\cite{FKS}, Theorem 1.3 and Theorem 1.6]\label{wpoincare}
Let $ D$ be an open bounded Lipschitz set in $\mathbb{R}^N$ and consider $1<p<\infty$ and a weight $\omega\in A_p$.

Then, there exist a positive constant $C( D)$ and $\delta>0$ such that for all $u\in H_0^1( D,\omega)$ and any $1\leq \sigma \leq\frac{N}{N-1}+\delta$ we have
\begin{equation}\label{poincareFabes}
\|u\|_{L^{\sigma p}( D,\omega dx)}\leq C( D)\|\nabla u\|_{L^p( D,\omega dx)},
\end{equation}
where $C( D)=c_{\omega} \mbox{diam}( D)| D|_{\omega}^{\frac{1}{p}\left(\frac{1}{\sigma}-1\right)}$ for a positive constant $c_\omega$ depending on $N,\ p$ and $\omega$.

Moreover for any $x_0\in\partial D$   there exist a positive constant $C=C(B_{\rho}(x_0))$ and $\delta>0$ such that $1\leq \sigma \leq\frac{N}{N-1}+\delta$ and any $u\in H^1( {{ D(x_0,\rho)}},\omega)$ vanishing on $\partial D\cap B_{\rho}(x_0)$ we have
\begin{equation*}
\|u\|_{L^{\sigma p}( D(x_0,\rho),\omega dx)}\leq C(B_{\rho})\|\nabla u\|_{L^p(( D(x_0,\rho),\omega dx)},
\end{equation*}
where $C(B_{\rho})= c_{  \omega}\rho^{\frac{N}{p}\left(\frac{1}{\sigma}-1\right) +1}$ for a positive constant $c_\omega$ depending on $\omega$, $N,\ p$ and $\xi$.
\end{theorem}

We want to apply such a Theorem to domains $D\subsetneq\mathscr{C}_{\Omega}\subset\mathbb{R}_{+}^{N+1}$ so that the correspondent exponent $\sigma$ relies to satisfy
$1\leq \sigma \leq\frac{N+1}{N}$.

 \medskip

As far as the weight is concerned, we set $\omega = y^{1-2s}$, that, actually,  belongs to $A_2$. Let us observe that, according to \cite{FKS}, there exists $\varepsilon_0>0$ such that   \eqref{poincareFabes} holds true with $p\geq2-\varepsilon_0$.

 \medskip

%Let us also observe that by definition $\partial_L\mathscr{C}_{\Omega}$ and $\partial\Omega$ have the same regularity. In particular notice that
%
%%On the other hand, it is clear that the boundary of the extension cylinder $\mathscr{C}_{\Omega}$ possesses, by its   definition, the same regularity as the boundary of $\Omega$, therefore, as we are considering $\Omega$ to be a smooth bounded domain, $\partial_L\mathscr{C}_{\Omega}$ satisfies the hypotheses above in Theorem \ref{wpoincare}. In fact, assuming that $\partial\Omega$ is a $\mathcal{C}^k$ manifold for some $k\geq1$,
%\begin{equation}\label{prev}
%\lim_{\rho\to0}\frac{|B_{\rho}(z)\backslash \Omega(z,\rho)|}{|B_{\rho}(z)|}=c \ , \qquad \forall z\in\partial\Omega\,,
%\end{equation}
% %More generally, we can consider domains $\Omega$ such that $\partial\Omega$ is a Lipstchiz manifold. In this case \eqref{prev} remains true replacing $\frac12$ with certain constant $0<c<1$.
% for some $0<c<1$. \footnote{ esta ultima observacion me parece que no se utiliza nunca: la quitamos? }

 \medskip

As an immediate consequence of Theorem \ref{wpoincare} we obtain the following result.
\begin{lemma}\label{poincarelemma}
Let $Z\in\Sigma_{\mathcal{D}}^*$ and $p\geq2-\varepsilon_0$ for some $\varepsilon_0>0$. Then, there exists $\overline{\rho}>0$, such that for any $\rho<\overline{\rho}$ and any $U\in \mathcal{X}_{\Sigma_{\mathcal{D}}}^s(\mathscr{C}_{\Omega})$ we have
\begin{equation}\label{poincare}
\|U\|_{L^{\sigma p}(\mathscr{C}_{\Omega}(Z,\rho),y^{1-2s}dxdy)}\leq c_s\rho|B_{\rho}|_{y^{1-2s}}^{\frac{1}{p}\left(\frac{1}{\sigma}-1\right)}\|\nabla U\|_{L^p(\mathscr{C}_{\Omega}(Z,\rho),y^{1-2s}dxdy)},
\end{equation}
with $1\leq \sigma \leq\frac{N+1}{N}+\delta$ for some $\delta>0$ and $c_s$ depending on $N$, $p$ and the weight $y^{1-2s}$.
\end{lemma}

Although Theorem \ref{holder_result} has been stated for Lipschitz domains, following  \cite{S}, we might prove  most of the results in this section under more general hypotheses on $\partial\Omega$. Then, we relax the smoothness hypotheses on $\partial\Omega$ and establish inequality \eqref{poincare} for functions in $\mathcal{X}_{\Sigma_{\mathcal{D},R}}^s(\mathscr{C}_{\Omega(z,R)})$ and, given some point $Z\in\mathscr{C}_{\Omega(z,R)}^\circ \backslash\Sigma_{\mathcal{D},R}^*$, also for functions in $H^1(\mathscr{C}_{\Omega}(Z,\rho),y^{1-2s}dxdy)$ vanishing on suitable sets.

\begin{definition}
Given $p\geq2-\varepsilon_0$ for some $\varepsilon_0\in (0,1)$ and an open bounded set $A$, we define $\mathcal{F}(\beta_s,A)$ as the family of sets $B\subset\overline A$ such that, for any $U\in H^1(A,y^{1-2s}dxdy)$ vanishing on $B$,
\begin{equation}\label{sobA}
\|U\|_{L^{\sigma p}(A,y^{1-2s}dxdy)}\leq \beta_{s}\mbox{diam}(A)|A|_{y^{1-2s}}^{\frac{1}{p}\left(\frac{1}{\sigma}-1\right)}\|\nabla U\|_{L^p(A,y^{1-2s}dxdy)},
\end{equation}
for some $\beta_s>0$ depending on $N$, $p$ and the weight $y^{1-2s}$, and  $1\leq \sigma \leq\frac{N+1}{N}+\delta$ for some $\delta>0$.
\end{definition}

With this scheme in mind, we focus first on finding   bounds for solutions to \eqref{nonhomo_abajo} in terms of the data of the problem.

%measure of the domain $\Omega(z,R)$, the datum $f$ and a positive constant $C=C(N,s,|\Sigma_{\mathcal{D}}|)$. This is done adapting to our framework \cite[Theorem B.2]{KS}.

%Next, we establish bounds on the oscillation of functions $w(x)$ satisfying \eqref{homo_abajo}. This is done using arguments similar to those of \cite[Theorem 8.5]{S} and \cite[Theorem D.5]{KS}.
%
%
%To accomplish this step we work with the extended problem \eqref{homo}. Gathering together these results, we will be able to prove the local H\"older regularity of solutions to problem \eqref{problema1}.\footnote{esta ultima frase, hace falta?}
%

\begin{theorem}\label{th:bound_nonhomo}
Let $u$ be a solution to \eqref{problema1} with $f\in L^{p}(\Omega)$, $p>\frac{N}{2s}$. Then, there exists a positive constant $C=C(N,s,|\Sigma_{\mathcal{D}}|)$ such that
\begin{equation*}
\|u\|_{L^{\infty}(\Omega)} \leq C\|f\|_{L^{p}(\Omega)} |\Omega|^{\frac{2s}{N}-\frac{1}{p}}.
\end{equation*}
\end{theorem}
In the proof of Theorem \ref{th:bound_nonhomo} we make use of the following technical result.
\begin{lemma}[\cite{KS},  Lemma B.1]\label{lem:B.1}
Let $\varphi(k)$ be a nonnegative and nonincreasing function defined for $k\geq k_0$ such that
\begin{equation*}
\varphi(h) \leq \frac{C_0}{(h-k)^{a}} \varphi^{b}(k), \qquad k<h,
\end{equation*}
where $C_0, a, b$ are positive constants with $b > 1$. Then, $\varphi(k_{0}+d)=0$,
with
$%\begin{equation*}%\label{B.3}
d^{a}= 2^{\frac{ab}{b-1}}C_0 |\varphi(k_{0})|^{b-1} .
$%\end{equation*}
\end{lemma}

\begin{proof}[Proof of Theorem \ref{th:bound_nonhomo}] Here we just prove  the upper bound, being the lower one completely analogous.
Let us take $k\geq0$, $U(x,y)=E_s[u (x) ]$ and $\psi=(U-k)_+\in \mathcal{X}_{\Sigma_{\mathcal{D}}}^s(\mathscr{C}_{\Omega})$ as a test function in \eqref{def:soldebil}.  Using the trace inequality \eqref{eq:traceineq} together with the H\"older inequality, we get
\begin{align*}
\kappa_s\int_{\mathscr{C}_{\Omega}}y^{1-2s}\nabla U\nabla\psi dxdy&=\kappa_s\int_{A_{+}^*(k)}y^{1-2s}|\nabla U|^2 dxdy=\int_{A_{+}(k)}(U(x,0)-k)f(x)dx\\
%&\leq\left(\int_{A_{+}(k)}|f|^2dx\right)^{\frac{1}{2}} \left(\int_{A_{+}(k)}\left|U(x,0)-k\right|^2dx\right)^{\frac{1}{2}}\\
&\leq \left(\int_{A_{+}(k)}\!\!\!\!|f|^2dx\right)^{\frac{1}{2}}\!\!\left(\! C_{\mathcal{D}}^{-1}|A_+(k)|^{\frac{2s}{N}}\!\int_{A_{+}^*(k)}\!\!\! y^{1-2s}|\nabla U|^2 dxdy\right)^{\frac{1}{2}}\!\!.
\end{align*}
Thus,
\begin{equation}\label{uno}
\begin{split}
\int_{A_+^*(k)}y^{1-2s}|\nabla U|^2 dxdy&\leq C_{\mathcal{D}}^{-1}\kappa_s^{-2}|A_+(k)|^{\frac{2s}{N}}\int_{A_+(k)}|f|^2dx
\leq \frac{\|f\|_{L^{p}(\Omega)}^{2}|A_+(k)|^{1-\frac{2}{p}+\frac{2s}{N}}}{C_{\mathcal{D}} \kappa_s^{2}},
\end{split}
\end{equation}
and applying   the trace inequality \eqref{eq:traceineq} to the left-hand side of \eqref{uno} we get   for any $h>k$,
\begin{equation*}
(h-k)^2|A_+(h)|^{\frac{2}{2_s^*}}\leq\left(\int_{A_+(k)}\left|U(x,0)-k\right|^{2_s^*}dx\right)^{\frac{2}{2_s^*}}\,.
\end{equation*}
Thus we deduce
\begin{equation*}
(h-k)^2|A_+(h)|^{\frac{2}{2_s^*}}\leq \frac{\|f\|_{L^{p}(\Omega)}^{2} }{(C_{\mathcal{D}}\kappa_s)^{2}}|A_+(k)|^{1-\frac{2}{p}+\frac{2s}{N}}\,,
\end{equation*}
and setting  $\varphi(h)=|A_+(h)|$, it follows that
\begin{equation*}
\varphi(h)\leq \frac{\|f\|_{L^{p}(\Omega)}^{2_s^*}}{(C_{\mathcal{D}}\kappa_s)^{2}} \
\frac{\varphi^{\left(1-\frac{2}{p}+\frac{2s}{N}\right)\frac{2_s^*}{2}}(k) }{(h-k)^{2_s^*}}.
\end{equation*}
Applying now Lemma \ref{lem:B.1} with $a=2_s^*$ and $b=\left(1-\frac{2}{p}+\frac{2s}{N}\right)\frac{2_s^*}{2}>1$, we find $|\varphi(k_0+d)|=0$ with $d=C(N,s,|\Sigma_{\mathcal{D}}|)\|f\|_{L^{p}(\Omega)}|\varphi(k_0)|^{\frac{b-1}{a}}$, and $\frac{b-1}{a}=\frac{2s}{N}-\frac{1}{p}$, i.e.
\begin{equation*}
U(x,0)\leq k_0+C(N,s,|\Sigma_{\mathcal{D}}|)\|f\|_{L^{p}(\Omega)}|A_+(k_0)|^{\frac{2s}{N}-\frac{1}{p}} \quad\mbox{a.e. in } \Omega,
\end{equation*}
for any $k_0\geq0$, and we conclude
$%\begin{equation*}
u(x)\leq  C(N,s,|\Sigma_{\mathcal{D}}|)\|f\|_{L^{p}(\Omega)}|\Omega|^{\frac{2s}{N}-\frac{1}{p}}\quad\mbox{a.e. in } \Omega.
$%\end{equation*}
\end{proof}

Let $v(x)$ be the solution to \eqref{nonhomo_abajo} and $V(x,y)=E_s[v(x)]$ the solution to \eqref{nonhomo_arriba}.
Since the function
$(V-k)_+\in \mathcal{X}_{\Sigma_{\mathcal{D}}}^s(\mathscr{C}_{\Omega})$ for any $k\geq0$, repeating the proof  above   we deduce that $\forall z \in \Omega$
\begin{equation}\label{bound_nonhomo_aux}
%\max\limits_{x\in\Omega(z,R)}
\| v(x) \|_{L^{\infty} (\Omega(z,R))} \leq C(N,s,|\Sigma_{\mathcal{D}}|)\|f\|_{L^p(\Omega)}|\Omega(z,R)|^{\frac{2s}{N}-\frac{1}{p}}.
\end{equation}
Now we turn our attention to the study of the behavior of solutions to the homogeneous problem \eqref{homo}.
\begin{lemma}[Caccioppoli inequality]\label{lem:caccioppoli}
Assume that $z_0\in\overline{\Omega}$ and $R>0$ and suppose  that the function $W\in \mathcal{X}_{\Sigma_{\mathcal{D},R}}^s(\mathscr{C}_{\Omega(z_0,R)})$ is a solution to problem \eqref{homo}. Then, for any $Z\in\mathscr{C}_{\Omega(z_0,R)}^{\circ}$ and $0<\rho<r<\widetilde{\rho}(Z)$,   we have that there exists $  C>0$ such that
\begin{equation*}
\int_{\mathscr{C}_{\Omega(z_0,R)}(Z,\rho)}\mkern-25mu y^{1-2s}|\nabla W|^2dxdy\leq \frac{C}{(r-\rho)^2}\int_{\mathscr{C}_{\Omega(z_0,R)}(Z,r)}\mkern-25mu y^{1-2s}|W|^2dxdy\,.
\end{equation*}
\end{lemma}
\begin{proof}
We use $\psi=\eta^2 W$ as a test function in \eqref{soldebil}, with $\eta\in C^1(\mathscr{C}_{\Omega(z_0,R)})$ such that it  vanishes on $\partial_L\mathscr{C}_{\Omega(z_0,R)}\backslash(\Sigma_{\mathcal{D},R}^*\cup\Sigma_{\mathcal{N},R}^*)$; observe that  in particular   $\psi\equiv 0$   on $\partial_L\mathscr{C}_{\Omega(z_0,R)}\backslash\Sigma_{\mathcal{N},R}^*$, so that we have that
\begin{equation}\label{Cacc}
\begin{array}{c}
\dys \int_{\mathscr{C}_{\Omega(z_0,R)}}\mkern-35mu y^{1-2s}\eta^2|\nabla W|^2dxdy=-2\int_{\mathscr{C}_{\Omega(z_0,R)}}\mkern-35mu y^{1-2s}\langle\eta\nabla W,W\nabla\eta\rangle dxdy\\
\dys \leq 2\left(\frac{1}{2\varepsilon}\int_{\mathscr{C}_{\Omega(z_0,R)}}\mkern-35mu y^{1-2s}|\nabla \eta|^2W^2dxdy+\frac{\varepsilon}{2}\int_{\mathscr{C}_{\Omega(z_0,R)}}\mkern-35mu y^{1-2s}\eta^2|\nabla W|^2dxdy\right),
\end{array}
\end{equation}
for any   $0<\varepsilon<1$.
To complete the proof, given $Z\in\mathscr{C}_{\Omega(z_0,R)}^\circ$ and $\rho<r<\widetilde{\rho}(Z)$ it is enough to set $\eta$ such that
\begin{equation*}
\eta\equiv1 \mbox{ in } B_{\rho}(Z),\qquad \eta\equiv0\mbox{ in } B_r^c(Z)\qquad  \mbox{ and }\qquad  |\nabla\eta|\leq\frac{c}{(r-\rho)}.
\end{equation*}
and plug into \eqref{Cacc}.
\end{proof}

Next we prove the following weighted version of the Poincar\'e Inequality.

\begin{lemma}\label{lem:psobolev}
Let $p\geq2-\varepsilon_0$ for some $0<\varepsilon_0<1$ and $U\in \mathcal{X}^s(\mathscr{C}_{\Omega})$ such that $\{U=0\}\in\mathcal{F}(\beta,A)$ for $A\subset\overline{\mathscr{C}}_{\Omega}$. Then $\exists \beta_s = \beta_s (N,p, y^{1-2s})>0$ such that
\begin{equation}\label{SobA}
\int_{A}y^{1-2s}|U|^pdxdy\leq \beta_s^p[\mbox{diam}(A)]^p|A|_{y^{1-2s}}^{\left(\frac{1}{\sigma}-1\right)}\  |\{(x,y)\in A: U\neq0\}|_{y^{1-2s}}^{\frac{1}{\sigma'}}\ \int_{A} y^{1-2s}|\nabla U|^pdxdy  ,
\end{equation}
and
\begin{equation}\label{level}
\int_{A_+^*(k,r)}\mkern-35mu y^{1-2s}|U-k|^2dxdy
\leq \beta_s^2r^2|B_r|_{y^{1-2s}}^{ \frac{1}{\sigma}-1 }|A_+^*(k,r)|_{y^{1-2s}}^{\frac{1}{\sigma'}}  \int_{A_+^*(k,r)}\mkern-35mu y^{1-2s}|\nabla U|^2dxdy ,
\end{equation}
with  $1\leq \sigma \leq\frac{N+1}{N}+\delta$ for some $\delta>0$.

%with   $1\leq \sigma \leq\frac{N+1}{N}+\delta$ for some $\delta>0$. %and $\frac{1}{\sigma}+\frac{1}{\sigma'}=1$, provided that
%{\color{red}\begin{itemize}
%\item[(i)] $k\in\mathcal{K}^+(Z)$ if $Z\in\Sigma_{\mathcal{D},R}^*$\,.
%\item[(ii)] $k\in\mathcal{K}^+(Z)$ is such that
%\begin{equation*}
%|A_+^*(k,r)|_{y^{1-2s}}\leq(1-\lambda)|\mathscr{C}_{\Omega(z,R)}(Z,r)|_{y^{1-2s}},
%\end{equation*}
%if $Z\in\mathscr{C}_{\Omega(z,R)}^{\circ}\backslash\Sigma_{\mathcal{D},R}^*$.
%\end{itemize}}
\end{lemma}
\begin{proof}

In fact, \eqref{SobA} is consequence of   \eqref{poincareFabes} and    the H\"older inequality.

As far as \eqref{level} is concerned, we   follows   \cite[Theorem 6.1]{S}: given $U\in \mathcal{X}^s(\mathscr{C}_{\Omega(z_0,R)})$, let us consider the function $t_k^+(U)=(U-k)_+$ that belongs to $\mathcal{X}^s(\mathscr{C}_{\Omega(z_0,R)})$ for any $k\in\mathbb{R}$. Moreover, if $U\in \mathcal{X}_{\Sigma_{\mathcal{D},R}}^s(\mathscr{C}_{\Omega(z_0,R)})$ then $t_k^+(U)\in \mathcal{X}_{\Sigma_{\mathcal{D},R}}^s(\mathscr{C}_{\Omega(z_0,R)})$ for any $k\geq0$.
Then, applying   \eqref{sobA} to $(U-k)_+$ with $p=2$, \eqref{level} follows.
\end{proof}

A direct consequence of Lemma \ref{lem:psobolev}  is the following result.

\begin{lemma}\label{lem:measure_lvlset}
Given $z_0\in\!\overline{\Omega}$ and $R>\!0$, let $U\in \mathcal{X}^s(\mathscr{C}_{\Omega(z_0,R)})$. Then, for any $Z\in\mathscr{C}_{\Omega(z_0,R)}^{\circ}$ and $0<r<\overline{\rho}(Z)$, there exist %$\varepsilon >0$,
$\varepsilon_0 \in (0,1) $ and $\beta_s=\beta_s (N,p,y^{1-2s})>0$ such that
\begin{equation}\label{level2}
\displaystyle (h-k)^2 |A_+^*(h,r)|_{y^{1-2s}}^{\frac{2}{q}}
\leq
 \beta_s^2r^2|B_r |_{y^{1-2s}}^{2\left(\frac{1}{q}-\frac{1}{p}\right)}
 |A_+^*(k,r)-A_+^*(h,r)|_{y^{1-2s}}^{\frac{2}{p}-1}
 \int_{A_+^*(k,r)}
 \mkern-35mu y^{1-2s}|\nabla U|^2dxdy,
\end{equation}
with $h>k$, $q=\frac{N+1}{N} (2-\varepsilon_0)$ and   $p=2-\varepsilon_0$.
% depending on $N$, $p$ and the weight $y^{1-2s}$; provided that
%\begin{itemize}
%\item[(i)] $k\in\mathcal{K}^+(Z)$ if $Z\in\Sigma_{\mathcal{D},R}^*$\,.
%\item[(ii)] $k\in\mathcal{K}^+(Z)$ is such that
%\begin{equation*}
%|A_+^*(k,r)|_{y^{1-2s}}\leq(1-\lambda)|\mathscr{C}_{\Omega(z,R)}(Z,r)|_{y^{1-2s}},
%\end{equation*}
%if $Z\in\mathscr{C}_{\Omega(z,R)}^{\circ}\backslash\Sigma_{\mathcal{D},R}^*$.
%\end{itemize}
\end{lemma}
\begin{proof}
Given $U\in \mathcal{X}^s(\mathscr{C}_{\Omega(z_0,R)})$ and $h>k$, let $t_{h,k}^+(U)=\{U\}^h-\{U\}^k$. Note that $t_{h,k}^+(U)\in \mathcal{X}^s(\mathscr{C}_{\Omega(z_0,R)})$ for any $k\in\mathbb{R}$. Moreover, if $U\in \mathcal{X}_{\Sigma_{\mathcal{D},R}}^s(\mathscr{C}_{\Omega(z_0,R)})$ then $t_{h,k}^+(U)\in\mathcal{X}_{\Sigma_{\mathcal{D},R}}^s(\mathscr{C}_{\Omega(z_0,R)})$ for any $h>k\geq0$.
%\begin{enumerate}
%\item[(i)] If $Z\in\Sigma_{\mathcal{D},R}^*$, then $\{t_{h,k}^+(U)=0\}=\{t_k^+(U)=0\}\supset\{U=0\}$ and, therefore, $\{t_{h,k}^+(U)=0\}\cap\mathscr{C}_{\Omega(z,R)}(Z,r)\in\mathcal{F}(\beta_s,\mathscr{C}_{\Omega(z,R)}(Z,r))$.
%\item[(ii)] If $Z\in\mathscr{C}_{\Omega(z,R)}^{\circ}\backslash\Sigma_{\mathcal{D},R}^*$, then $\{t_{h,k}^+(U)=0\}\supset\{t_{k}^+(U)=0\}$. Repeating the arguments for (ii) in Lemma \ref{lem:sobolev_Truncate} we conclude $\{t_{h,k}^+(U)=0\}\cap\mathscr{C}_{\Omega(z,R)}(Z,r)\in\mathcal{F}(\beta_s,\mathscr{C}_{\Omega(z,R)}(Z,r))$.
%\end{enumerate}
Thus, we use  Lemma \ref{lem:psobolev} with $\sigma=\frac{N+1}{N}$ and $p=2-\varepsilon_0$ so that taking $q=\sigma p=\frac{N+1}{N}(2-\varepsilon_0)$ we obtain,
\begin{equation}\label{A}
\left(\int_{\mathscr{C}_{\Omega(z_0 ,R)}(Z,r)}\mkern-35mu y^{1-2s}|t_{h,k}^+(U)|^q dxdy\right)^{\frac{1}{q}}\mkern-10mu
\leq \beta_sr|B_r |_{y^{1-2s}}^{\frac{1}{q}-\frac{1}{p}}\left(\int_{A_+^*(k,r)-A_+^*(h,r)}\mkern-35mu y^{1-2s}
|\nabla U|^pdxdy\right)^{\frac{1}{p}}.\mkern-35mu.
\end{equation}
At one hand, it is immediate that
\begin{equation}\label{B}
(h-k)^2|A_+^*(h,r)|_{y^{1-2s}}^{\frac{2}{q}}\leq \left(\int_{\mathscr{C}_{\Omega(z_0 ,R)}(Z,r)}y^{1-2s}\left|t_{h,k}^+(U)\right|^q dxdy\right)^{\frac{2}{q}}.
\end{equation}
On the other hand, thanks to H\"older inequality
\begin{equation}\label{C}
\left(\int_{A_+^*(k,r)-A_+^*(h,r)}\mkern-35mu y^{1-2s}|\nabla U|^pdxdy\right)^{\frac{2}{p}}\leq
 |A_+^*(k,r)-A_+^*(h,r)|_{y^{1-2s}}^{\frac{2}{p}-1} \
\int_{A_+^*(k,r)}\mkern-35mu y^{1-2s}|\nabla U|^2dxdy .
\end{equation}
Thus \eqref{level2} follows by gathering together \eqref{A}, \eqref{B} and \eqref{C}.
\end{proof}

Following \cite[Theorem 8.1]{S}, we show the next result.

\begin{theorem}\label{thcota} Let $z_0\in\!\overline{\Omega}$,  $R>\!0$,   and let $W\in \mathcal{X}_{\Sigma_{\mathcal{D},R}}^s(\mathscr{C}_{\Omega(z_0,R)})$ be  a solution to
 the homogeneous problem \eqref{homo}. Then, for any $Z\in\mathscr{C}_{\Omega(z_0,R)}^{\circ}$, $0<\ell<1$ and
 $0<r<\min\{\widetilde{\rho}(Z),\overline{\rho}(Z)\}$, there exists a positive constant $\Lambda=\Lambda(\ell)$ such that
\begin{equation*}
|A_+^*(k+\ell d,r-\ell r)|=0, \quad \mbox{with }  k\in\mathcal{K}^+(Z) \qquad \mbox{and} \qquad |A_-^*(k-\ell d,r- \ell r)|=0, \qquad \mbox{with }  k\in\mathcal{K}^-(Z)    \,,
\end{equation*}
where
\begin{equation}\label{de}
d^2\geq \frac{1}{\Lambda (\ell) \ |B_{r} |_{y^{1-2s}}}\int_{A_+^*(k,r)}\mkern-35mu y^{1-2s}|W-k|^2dxdy\,.
\end{equation}
\end{theorem}

In the proof of Theorem \ref{thcota} we make use of the following technical result.

\begin{lemma}[\cite{KS}, Lemma C.7]\label{lem:C.7}
Assume that $\varphi(k,\rho)$ is a nonnegative function defined for $k\geq k_0$ and $0<\rho\leq r_0$ which is nonincreasing with respect to $k$, nondecreasing with respect to $\rho$ and such that
\begin{equation*}%\label{C.8}
\varphi(h,\rho) \leq \frac{C_0}{(h-k)^{\alpha}(r-\rho)^{\gamma}}\varphi^{\mu}(k,r),\qquad k<h,\ \rho<r\leq r_0,
\end{equation*}
where $C, \alpha, \beta, \gamma$ are positive constants with $\mu>1$. Then there exist $\ell \in (0,1) $ and $d>0$ such that  $\varphi(k_0+\ell d,r_0(1- \ell) )=0$, with
\begin{equation*}%\label{C.10}
d^{\alpha}= C_0 \frac{2^{(\alpha+\gamma)\frac{\mu}{\mu -1}} [\varphi(k_{0},r_{0})]^{\mu-1}}{\ell^{\alpha+\gamma}r_{0}^{\gamma}}.
\end{equation*}
\end{lemma}

\begin{proof}[Proof of Theorem \ref{thcota}]
Given $z_0\in\overline \Omega $, $k_0\in\mathcal{K}^+(z_0)$% [{\color{red}satisfying \eqref{hyp:thcota}}]
 and $k\geq k_0$, let us define
\begin{equation*}%\label{uiter}
i(k,\rho)=\int_{A^*(k,\rho)}y^{1-2s}|W-k|^2dxdy
\quad \mbox{and} \qquad
a(k,\rho)=|A_+^*(k,\rho)|_{y^{1-2s}}.
\end{equation*}
Observe that for $h>k$ we have
\begin{equation}\label{ineq_hk}
 (h-k)^2|A_+^*(h,\rho)|_{y^{1-2s}}\leq\int_{A_+^*(k,r)}\mkern-20mu y^{1-2s}|W-k|^2dxdy.
\end{equation}

\smallskip

Assume that $Z\in\Sigma_{\mathcal{D},R}^* \cap \mathscr{C}_{\Omega(z_0 ,R)}$ and let $0<r_0<\min\{\widetilde{\rho}(Z),\overline{\rho}(Z)\}$. Then, due to Lemma \ref{lem:caccioppoli}  and Lemma \ref{lem:psobolev}, for any $r_0(1-\ell)  \leq\rho<r\leq r_0$ and $h>k$, we have
\begin{align}\label{ineq:prev_holder}
\int_{A_+^*(h,\rho)}\mkern-35mu y^{1-2s}|W-h|^2dxdy&\leq K_{\mathscr{C}_{\Omega}(\rho)}\left(\int_{A_+^*(h,\rho)}\mkern-35mu y^{1-2s}|\nabla W|^2dxdy\right)|A_+^*(h,\rho)|_{y^{1-2s}}^{\frac{1}{\sigma'}}\nonumber\\
&\leq K_{\mathscr{C}_{\Omega}(\rho)}\left(\int_{A_+^*(k,\rho)}\mkern-35mu y^{1-2s}|\nabla W|^2dxdy\right)|A_+^*(k,\rho)|_{y^{1-2s}}^{\frac{1}{\sigma'}}\\
&\leq K_{\mathscr{C}_{\Omega}(\rho)}\left(\frac{1}{(r-\rho)^2}\int_{A_+^*(k,r)}\mkern-35mu y^{1-2s}|W-k|^2dxdy\right)|A_+^*(k,r)|_{y^{1-2s}}^{\frac{1}{\sigma'}},\nonumber
\end{align}
where $K_{\mathscr{C}_{\Omega}(r)}= \beta_s^2r^2|B_{r}|_{y^{1-2s}}^{\frac{1}{\sigma}-1}$, with $\beta_s= \beta_s (N, y^{1-2s}, \partial \Omega) >0$
%depending on $N$, the weight $y^{1-2s}$ and the geometry of $\partial\Omega$
and  $1\leq \sigma \leq\frac{N+1}{N}+\delta$ for some $\delta>0$.

\smallskip

Assume,  on the contrary,   that $Z_0\in\mathscr{C}_{\Omega(z_0 ,R)}^{\circ}\backslash\Sigma_{\mathcal{D},R}^*$.
Recalling  \eqref{ineq:zetaw}, let $\Lambda=\Lambda(\ell)>0$ satisfying
\begin{equation*}
\frac{\Lambda}{\zeta_s(1-\ell)^{N+2(1-s)}}\leq(1-\lambda) \qquad \mbox{ for some  } \lambda \in (0,1).
\end{equation*}
  Therefore, given $h\geq k_0$ and $(1-\ell) r_0\leq \rho\leq r_0$, we find
  $$
\begin{array}{c}
\dys |A_+^*(h,\rho)|_{y^{1-2s}} \leq|A_+^*(k_0,r_0)|_{y^{1-2s}}\leq |\mathscr{C}_{\Omega(z_0,R)}(Z,r_0)|_{y^{1-2s}}\leq |B_{r_0}(Z)|_{y^{1-2s}}
\\[1.5 ex] \dys
 \leq \frac{|B_{\rho}(Z)|_{y^{1-2s}} }{(1-\ell)^{N+2(1-s)}}
\leq\frac{\Lambda \ |\mathscr{C}_{\Omega(z_0 ,R)}(Z,\rho)|_{y^{1-2s}} }{\zeta_s(1-\ell)^{N+2(1-s)}} \leq (1-\lambda)|\mathscr{C}_{\Omega(z_0 ,R)}(Z,\rho)|_{y^{1-2s}} .
\end{array}
$$
Using Lemma   \ref{lem:caccioppoli} and Lemma \ref{lem:psobolev} we deduce that  \eqref{ineq:prev_holder} holds true.

As a consequence, for any $Z\in\mathscr{C}_{\Omega(z_0,R)}^\circ$,
\begin{equation}\label{iter}
\displaystyle  i(h,\rho)\leq  \frac{K_{\mathscr{C}_{\Omega}(\rho)} }{(r-\rho)^2} i(k,r)[a(k,r)]^{\frac{1}{\sigma'}}, \qquad  r_0(1-\ell)\leq\rho<r\leq r_0, \mbox{ and } \  h>k\geq k_0
\end{equation}
with $k_0\in\mathcal{K}^+(Z)$ satisfying \eqref{hyp:thcota}. Moreover, since   $|B_{\mu r} |_{y^{1-2s}}=\mu^{N+2(1-s)}|B_{r} |_{y^{1-2s}}$,
we have that $\displaystyle K_{\mathscr{C}_{\Omega}(\mu r)}=\mu^{\varsigma}K_{\mathscr{C}_{\Omega}(r_0)}$, where  $\varsigma=2+\left(\frac{1}{\sigma}-1\right)\left(N+2(1-s)\right)$.

If we let $1<\sigma\leq1+\frac{2}{N-2s}$ (so that $\varsigma>0$)  then $K_{\mathscr{C}_{\Omega}(r)}\leq K_{\mathscr{C}_{\Omega}(r_0)}$ for any  $0<r<r_0$. Hence, from \eqref{iter}, we obtain
\begin{equation}\label{iterb}
\displaystyle  i(h,\rho)\leq  \frac{K_{\mathscr{C}_{\Omega}(r_0)} }{(r-\rho)^2} i(k,r)[a(k,r)]^{\frac{1}{\sigma'}},\quad \rho<r\leq r_0,\ h>k\geq k_0,\
\end{equation}
with $K_{\mathscr{C}_{\Omega}(r_0)}=\beta_s^2r_0^2|B_{r_0}|_{y^{1-2s}}^{\frac{1}{\sigma}-1}$.
We set now $\xi+1=\theta\xi$ and $\frac{\xi}{\sigma'}=\theta$, so that $\theta=\frac12 +\sqrt{\frac14 +\frac{1}{\sigma'}}>1$ turns out to be  the unique  positive solution to the equation $\theta^2-\theta-\frac{1}{\sigma'}=0$. Assume in addition that the constant $\Lambda$ satisfies
\begin{equation}\label{lambda2}
\Lambda^{\frac{\theta}{2}}\leq \frac{\ell^{ \xi+1 }}{\beta_s^{\xi}2^{(\xi+1)\frac{\theta}{\theta-1}}}.
\end{equation}
From \eqref{ineq_hk} and \eqref{iterb}, we obtain
\begin{equation*}
|i(h,\rho)|^{\xi}|a(h,\rho)|
\leq\frac{K_{\mathscr{C}_{\Omega}(r_0)}^{\xi}}{(r-\rho)^{2\xi}(h-k)^{2 }} |i(k,r)|^{\xi+1}|a(k,r)|^{\frac{\xi}{\sigma'}}.
\end{equation*}
Then, taking $\varphi(k,\rho)=|i(k,\rho)|^{\xi}|a(k,\rho)|$, it follows that $\varphi$ satisfies
\begin{equation*}
\varphi(h,\rho)\leq \frac{K_{\mathscr{C}_{\Omega}(r_0)}^{\xi}}{(r-\rho)^{2\xi}(h-k)^{2}}\varphi^{\theta}(k,r),\quad h>k\geq k_0,\ \rho<r\leq r_0.
\end{equation*}
Using Lemma \ref{lem:C.7} with $\alpha=2$, $\mu=\theta$, $\gamma=2\xi$, we deduce that   exist $d_0>0$ and $\ell\in (0,1) $ such that
\begin{equation*}
\varphi(k_0+\ell d_0,r_0(1-\ell) )=0,
\end{equation*}
for any $k_0\in\mathcal{K}^+(Z)$ satisfying \eqref{hyp:thcota}, $0<r_0<\min\{\widetilde{\rho}(Z),\overline{\rho}(Z)\}$ and  $d_0$ such that
%Since $\varphi(k,\rho)=|i(k,\rho)|^{\xi}|a(k,\rho)|$ and $\xi\frac{\theta-1}{2}=\sigma'\theta\frac{\theta-1}{2}=\frac{\sigma'}{2}\frac{1}{\sigma'}=\frac{1}{2}$
\begin{align*}
d_0 & = \frac{2^{\frac{(\xi+1)\theta}{\theta-1}}}{\ell^{\xi+1}}\frac{K_{\mathscr{C}_{\Omega}(r_0)}^{\xi/2}[\varphi(k_0,r_0)]^\frac{\theta-1}{2}}{r_0^{\xi}}
%=
%\frac{2^{\frac{(\xi+1)\theta}{\theta-1}}}{l^{\xi+1}}\left(\beta_s^2r_0^2|B_{r_0}(Z)|_{y^{1-2s}}^{\frac{1}{\sigma}-1}\right)^{\xi/2}
%\frac{[\varphi(k_0,r_0)]^\frac{\theta-1}{2}}{r_0^{\xi}}
%\\
%&= \frac{2^{\frac{(\xi+1)\theta}{\theta-1}}\beta_s^{\xi}}{l^{\xi+1}}\frac{[\varphi(k_0,r_0)]^\frac{\theta-1}{2}}{|B_{r_0}(Z)|_{y^{1-2s}}^{\frac{\xi}{2
%\sigma'}}}
%=\frac{2^{\frac{(\xi+1)\theta}{\theta-1}}\beta_s^{\xi}}{l^{\xi+1}}\frac{[\varphi(k_0,r_0)]^\frac{\theta-1}{2}}{|B_{r_0}(Z)|_{y^{1-2s}}^{\frac{\theta}{2}}}
%\\
%&
%=\frac{2^{\frac{(\xi+1)\theta}{\theta-1}}\beta_s^{\xi}}{l^{\xi+1}|B_{r_0}(Z)|_{y^{1-2s}}^{\frac{1}{2}}}\left|\frac{\varphi(k_0,r_0)}{|B_{r_0}(Z)|_{y^{1-2s}}}\right|^{\frac{\theta-1}{2}}\\
%%&=\frac{2^{\frac{(\xi+1)\theta}{\theta-1}}}{l^{\xi+1}}\beta_s^{\xi}\left(\frac{1}{|B_{r_0}(Z)|_{y^{1-2s}}}|i(k_0,r_0)|\right)^{\frac12}\left(\frac{a(k_0,r_0)}{|B_{r_0}(Z)|_{y^{1-2s}}}\right)^{\frac{\theta-1}{2}}\\
%&=\frac{2^{\frac{(\xi+1)\theta}{\theta-1}}}{l^{\xi+1}}\beta_s^{\xi}\left(\frac{1}{|B_{r_0}(Z)|_{y^{1-2s}}}\int_{A_+^*(k,r_0)}\mkern-35mu y^{1-2s}|W-k_0|^2dxdy\right)^{\frac{1}{2}}\!\!\left(\frac{|A_+^*(k_0,r_0)|_{y^{1-2s}}}{|B_{r_0}(Z)|_{y^{1-2s}}}\right)^{\frac{\theta-1}{2}}\\
%&=\frac{2^{\frac{(\xi+1)\theta}{\theta-1}}}{l^{\xi+1}}\beta_s^{\xi}\Lambda^{\frac{\theta}{2}}\left(\frac{1}{\Lambda|B_{r_0}(Z)|_{y^{1-2s}}}\int_{A_+^*(k,r_0)}\mkern-35mu y^{1-2s}|W-k_0|^2dxdy\right)^{\frac{1}{2}}\\
 \geq \left(\frac{1}{\Lambda|B_{r_0}|_{y^{1-2s}}}\int_{A_+^*(k,r_0)}\mkern-35mu y^{1-2s}|W-k_0|^2dxdy\right)^{\frac{1}{2}}.
\end{align*}
Since $|A_+^*(k_0+\ell d_0,r_0(1-\ell) )|_{y^{1-2s}}=0$ implies $|A_+^*(k_0+\ell d_0,r_0(1-\ell) )|=0$ the proof is complete.

\smallskip

The proof on the lower bound follows using the same inequalities on $(W+k)^-$ and getting the bounds on
$|A_-^*(k_0-\ell d,r_0(1-\ell)  )|_{y^{1-2s}}$.
\end{proof}
%Using the function $t_k^-(U) =(U-k)_-$  and repeating the arguments above in Lemma \ref{lem:caccioppoli}, Lemma \ref{lem:sobolev_Truncate} and Theorem \ref{thcota} we can establish the following.

%\begin{theorem}\label{thcota_inferior} Given $z\in\!\overline{\Omega}$ and $R>\!0$, assume that $\mathscr{C}_{\Omega(z,R)}$ is a $\lambda$-admissible set and let $W\in \mathcal{X}_{\Sigma_{\mathcal{D},R}}^s(\mathscr{C}_{\Omega(z,R)})$ be a solution to the homogeneous problem \eqref{homo}. Then, for any $Z\in\mathscr{C}_{\Omega(z,R)}^{\circ}$, $0<l<1$ and $0<r<\min\{\widetilde{\rho}(Z),\overline{\rho}(Z)\}$, there exist a positive constant $\Lambda=\Lambda(l)$ such that
%\begin{equation*}
%|A_-^*(k-ld,r-lr)|=0,
%\end{equation*}
%where
%\begin{equation*}
%d^2\geq \frac{1}{\Lambda|B_{r}|_{y^{1-2s}}}\int_{A_+^*(k,r)}\mkern-35mu y^{1-2s}|W-k|^2dxdy,
%\end{equation*}
%provided that $k\in\mathcal{K}^-(Z)$ is such that
%\begin{equation*}
%|A_-^*(k,r)|_{y^{1-2s}}\leq\Lambda|\mathscr{C}_{\Omega(z,R)}(Z,r)|_{y^{1-2s}}.
%\end{equation*}
%\end{theorem}

%As a consequence of the two former results, we obtain an $L^{\infty}$ bound on solutions to problem \eqref{homo_abajo}.

As a consequence of the above Theorem we get the $L^{\infty}$ bound on $W$.

\begin{corollary}\label{cor:infty_boundhomo}
Let $z_0\in\!\overline{\Omega}$,  $R>\!0$,   and let  $W\in \mathcal{X}_{\Sigma_{\mathcal{D},R}}^s(\mathscr{C}_{\Omega(z_0 ,R)})$ be a solution to the homogeneous problem \eqref{homo}; consider the set $\mathscr{C}_{\Omega(z_0,R/2)}^{m}\!\!=\!\mathscr{C}_{\Omega(z_0,R/2)}\cap\{y\!<\! m\}$ with $m>0$.
Then, $W\!\in\! L^{\infty}( {\mathscr{C}_{\Omega(z_0,R/2)}^{m}})$ for any   $m>0$. \\
In particular, any solution  $w\in H_{\Sigma_{\mathcal{D},R}}^s(\Omega(z_0 ,R))$ of problem \eqref{homo_abajo},
satisfies $w\in L^{\infty}( {\Omega(z_0,R/2)})$.
\end{corollary}

\begin{proof}
First, let us prove that $w\in L^{\infty}( {\Omega(z_0,R/2)})$ with $w$ satisfying problem \eqref{homo_abajo}. Let $W\!\in\! \mathcal{X}_{\Sigma_{\mathcal{D},R}}^s(\mathscr{C}_{\Omega(z_0 ,R)})$ a solution to problem \eqref{homo} %so that $w\!=\! W(x,0)$ satisfies \eqref{homo_abajo}. S
and since $ {\Omega(z_0,R/2)}$ is a   bounded set, there exists $Z_i\!=\!(z_i,0)\!\in\!\mathscr{C}_{\Omega(z_0 ,R)}^{\circ}$, $i=1,2,\ldots, M$ such that
\begin{equation}\label{intersection}
\overline{\Omega(z_0,R/2)}=\left(\bigcup\limits_{i=1}^{M}\mathscr{C}_{\Omega(z_0 ,R)}^{\circ}(Z_i,r_i/2)\right)\cap\{y=0\},
\end{equation}
with $0<r_i<\{\widetilde{\rho}(Z_i),\overline{\rho}(Z_i)\}$. Let $\overline{k}>0$ and $\hat{k}<0$ be such that,
\begin{align*}
|A_+^*(\overline{k},r_i)| \leq\Lambda|\mathscr{C}_{\Omega(z_0 ,R)}(Z_i,r_i)|,\qquad \mbox{ and } \qquad
|A_-^*(\hat{k},r_i)| \leq\Lambda|\mathscr{C}_{\Omega(z_0 ,R)}(Z_i,r_i)|,
\end{align*}
for any $i=1,2,\ldots,M$. Then, applying Theorem \ref{thcota}   we conclude that, given $X\in\mathscr{C}_{\Omega(z_0 ,R)}(Z_i,r_i)$ for some $i=1,2,\ldots,M$; we have
\begin{equation}\label{ineq:bound_infty}
\kappa_m:=\hat{k}-\ell d\leq W(x,y)\leq\kappa_{M}:=\overline{k}+\ell d,
\end{equation}
with
\begin{equation*}
d^2\geq \frac{1}{\Lambda|B_{r}|_{y^{1-2s}}}\int_{\mathscr{C}_{\Omega(z_0 ,R)}}\mkern-15mu y^{1-2s}|W|^2dxdy,
\end{equation*}
for any $0<r<\min\limits_{i=1,\ldots,M}r_i$. In particular, by  \eqref{intersection}, the former inequality holds for any point $X=(x,0)$ with $x\in\overline{\Omega(z_0,R/2)}$ and we are done.

\smallskip

As $\mathscr{C}_{\Omega(z_0,R/2)}$ is an unbounded domain, if we repeat the steps above in order to prove that $W\in L^{\infty}(\overline{\mathscr{C}}_{\Omega(z_0,R/2)})$ from \eqref{ineq:bound_infty}, the numbers $\hat{k}, \overline{k}$ do diverge when considering a covering sequence $\{Z_i\}_{i\in\mathbb{N}}$. Nevertheless, it is clear that given any finite truncation of the extension cylinder, $\mathscr{C}_{\Omega(z_0,R/2)}^{m}=\mathscr{C}_{\Omega(z_0,R/2)}\cap\{y<m\}$, there exists a finite covering sequence and hence, we conclude $W\in L^{\infty}(\overline{\mathscr{C}^{m}_{\Omega(z_0,R/2)}})$ for all finite $m>0$.
\end{proof}

We focus now on the oscillation of the solutions $W\in \mathcal{X}_{\Sigma_{\mathcal{D},R}}^s(\mathscr{C}_{\Omega(z_0 ,R)})$ to problem \eqref{homo}. Let us set
\begin{equation*}
m(\rho)=\inf\limits_{X\in \overline{\mathscr{C}}_{\Omega(z_0 ,R)}(Z,\rho)}\mkern-10mu W(X)\qquad\qquad \mbox{and}\qquad\qquad M(\rho)=\sup\limits_{X\in\overline{\mathscr{C}}_{\Omega(z_0 ,R)}(Z,\rho)}\mkern-10mu W(X).
\end{equation*}
and define the oscillation function as
\begin{equation*}
\omega(\rho):=M(\rho)-m(\rho).
\end{equation*}

Our aim is to give some estimates on $\omega(\rho)$ through the following result.

\begin{theorem}\label{th:pre_holder}
Given $z_0 \in\overline{\Omega}$ and $R>0$, let $Z\in\mathscr{C}_{\Omega(z_0 ,R)}^{\circ}$  and let $W\in \mathcal{X}_{\Sigma_{\mathcal{D},R}}^s(\mathscr{C}_{\Omega(z_0 ,R)})$ be a solution to the homogeneous problem \eqref{homo}. Moreover, given $0<4\rho<\min\{\widetilde{\rho}(Z),\overline{\rho}(Z)\}$ let $0<\eta<1$ such that,
\begin{itemize}
\item[(i)] $(M(4\rho)-\eta\omega(4\rho),+\infty)\subset\mathcal{K}^+(Z)$,
\item[(ii)] $|A_+^*(M(4\rho)-\eta\omega(4\rho),2\rho)|_{y^{1-2s}}\leq\Lambda|\mathscr{C}_{\Omega(z_0 ,R)}(Z,2\rho)|_{y^{1-2s}}$,
\end{itemize}
where $\Lambda$ is determined by \eqref{de} with $\ell=\frac{1}{2}$. Then, there exists $0<\overline{\eta}<1$ independent from  $Z$ and $\rho$ such that,
\begin{equation}\label{ineq:oscillation}
\omega(\rho)\leq \overline{\eta}\omega(4\rho).
\end{equation}
\end{theorem}

\begin{proof}
Let  $Z\in\mathscr{C}_{\Omega(z_0,R)}^\circ$ and $0<4\rho<\min\{\widetilde{\rho}(Z),\overline{\rho}(Z)\}$,  let us define the sequence
\begin{equation*}
k_j=M(4\rho)-\eta_j\omega(4\rho),\quad\mbox{with }\eta_{j}=\frac{1}{2^{j+1}},\ j \in \mathbb{N}.
\end{equation*}
Assume first that $Z\in\mathscr{C}_{\Omega(z_0,R)}^\circ \backslash\Sigma_{\mathcal{D},R}^*$ so that $\mathcal{K}^+(Z)=(-\infty,\infty)$ and observe that  one of the following conditions is satisfied: either
\begin{equation}\label{ineq:choice}
|A_+^*(k_0,2\rho)|_{y^{1-2s}} \leq\frac12 |\mathscr{C}_{\Omega(z_0,R)}(Z,2\rho)|_{y^{1-2s}}
\quad \mbox{or} \quad
|A_-^*(k_0,2\rho)|_{y^{1-2s}} \leq \frac12 |\mathscr{C}_{\Omega(z_0 ,R)}(Z,2\rho)|_{y^{1-2s}}.\nonumber
\end{equation}
Assume without loss of generality that $|A_+^*(k_0,2\rho)|\leq\frac12 |\mathscr{C}_{\Omega(z_0,R)}(Z,2\rho)|$. As a consequence, $$|A_+^*(k_j,2\rho)|\leq\frac12 |\mathscr{C}_{\Omega(z_0,R)}(Z,2\rho)| \qquad \mbox{ for } \qquad j\geq1. 
$$

On the other hand, if $Z\in\Sigma_{\mathcal{D},R}^*$, we can assume that at least one between $M(4\rho)$ and $-m(4\rho)$ is greater than $\frac{1}{2}\omega(4\rho)$; suppose that $M(4\rho)>\frac{1}{2}\omega(4\rho)$. Therefore we have that  $k_j>0$ for $j\geq0$.

Then, using Lemma \ref{lem:measure_lvlset} with $h=k_{j+1}$ and $k=k_j$, we obtain
\begin{align*}
 (k_{j+1}-k_{j})^2|A_+^*(k_{j+1},2\rho)|_{y^{1-2s}}^{\frac{2}{q}}&\leq \beta_s^2(2\rho)^2|B_{2\rho}|_{y^{1-2s}}^{2\left(\frac{1}{q}-\frac{1}{p}\right)}\int_{A_+^*(k_{j},2\rho)}\mkern-20mu y^{1-2s}|\nabla W|^2dxdy\,,
\end{align*}
with $p,q$ such that $q=\frac{N+1}{N} (2-\varepsilon_0)$ and   $p=2-\varepsilon_0$ for a suitable $ \varepsilon_0>0$.

Moreover, applying Lemma \ref{lem:caccioppoli} to the function $t_{k_j}^+(W)\in \mathcal{X}_{\Sigma_{\mathcal{D},R}}^s(\mathscr{C}_{\Omega(z,R)})$, $j\geq0$, we find
\begin{align*}
\int_{A_+^*(k_{j},2\rho)}\mkern-20mu y^{1-2s}|\nabla W|^2dxdy\leq
\frac{C}{4\rho^2}\int_{A_+^*(k_{j},4\rho)}\mkern-20mu y^{1-2s}|W-k_{j}|^2dxdy
 \leq \frac{C}{4\rho^2} [M(4\rho)-k_j]^2|B_{4\rho}(Z)|_{y^{1-2s}}.\nonumber
\end{align*}
Gathering together the above inequalities we have that
\begin{equation}\label{ineq:msr}
(k_{j+1}-k_j)^2|A_+^*(k_{j+1},2\rho)|_{y^{1-2s}}^{\frac{2}{q}}
\leq  C\beta_s |B_{2\rho}|_{y^{1-2s}}^{2\left(\frac{1}{q}-\frac{1}{p}\right)+1}[M(4\rho)-k_j]^2
   |A_+^*(k_j,2\rho)-A_+^*(k_{j+1},2\rho)|_{y^{1-2s}}^{\frac{2}{p}-1},
\end{equation}
where the constant $C>0$ is the one appearing in the Caccioppoli inequality. Let us define
\begin{equation*}
\varphi(k)=\frac{|A_+^*(k,2\rho)|_{y^{1-2s}}}{|\mathscr{C}_{\Omega(z,R)}(Z,2\rho)|_{y^{1-2s}}},
\end{equation*}
and note that, by  \eqref{ineq:zeta} and \eqref{ineq:zetaw}, we have $|B_{2\rho}|_{y^{1-2s}}\leq\frac{1}{\zeta_s}|\mathscr{C}_{\Omega(z,R)}(Z,2\rho)|_{y^{1-2s}}$. Then, since $2\left(\frac{1}{q}-\frac{1}{p}\right)+1>0$, taking into account that
$$k_{j+1}-k_{j}=\eta_{j+1}\omega(4\rho) \qquad \mbox{ and } \qquad M(4\rho)-k_j=\eta_j\omega(4\rho)\,,
$$ from \eqref{ineq:msr} we find
\begin{equation*}
|\varphi(k_{j+1})|^{\frac{2}{q}}\leq\vartheta \left[\varphi(k_j)-\varphi(k_{j+1})\right]^{\frac{2}{p}-1} \qquad \mbox{with }
 \quad \vartheta=\frac{4C\beta_s}{\zeta_{s}^{2\left(\frac{1}{q}-\frac{1}{p}\right)+1}}\,.
\end{equation*}
 Let us set $\displaystyle\mu=\frac{2}{q}\frac{1}{\frac{2}{p}-1}>0$ and $a=\frac{p}{2-p}$, so that the above inequality turns into
\begin{equation*}
\varphi^{\mu} (k_{n}) \leq \vartheta^{a} \left[\varphi(k_j)-\varphi(k_{j+1}) \right], \qquad j \geq 0 \,.
\end{equation*}
Summing up the above inequality  for $j=0,1,\ldots,n$ and noticing that $\varphi(k_{j})\geq\varphi(k_{n})$ we get
\begin{equation*}
n\varphi^{\mu} (k_{n})\leq \vartheta^{a}\left[\varphi(k_0)-\varphi(k_{n+1}) \right]\,,
\end{equation*}
and  by  \eqref{ineq:choice}, we conclude that
\begin{equation}\label{ineq:toadm}
\varphi(k_{n})\leq \left(\frac{\vartheta^{a}\varphi(k_0)}{n}\right)^{\frac{1}{\mu}}. %\leq \left(\frac{\vartheta^{a}(1-\lambda)}{n}\right)^{\frac{1}{\mu}}.
\end{equation}
Let us set $\overline{n}>0$ such that
\begin{equation}\label{n_indep}
\overline{n}\geq %\left\lceil\frac{\vartheta^{a}(1-\lambda)}{\Lambda^{\mu}}\right\rceil=
\left\lceil\frac{(4C\beta_s)^a \varphi (k_0)}{\zeta_s^{\mu-1}\Lambda^{\mu}}\right\rceil,
\end{equation}
where $\Lambda$ is determined by \eqref{hyp:thcota} with $\ell=\frac{1}{2}$, $\zeta_s$ depends on $\zeta$ in \eqref{ineq:zeta} and the \textit{$A_2$-constant} (see \eqref{ineq:zetaw}), the constant $\beta_s$ depends on $N$ and the weight $y^{1-2s}$ and $C>0$ is an universal constant coming from the Caccioppoli inequality.\newline
Consequently, $\overline{n}$ is independent of $Z$ and $\rho$. Then, by  inequality \eqref{ineq:toadm}, we find
\begin{equation*}
\frac{|A_+^*(k_n,2\rho)|_{y^{1-2s}}}{|\mathscr{C}_{\Omega(z,R)}(Z,2\rho)|_{y^{1-2s}}}\leq \Lambda ,\qquad\forall n\geq\overline{n}.
\end{equation*}
Applying Theorem \ref{thcota} with $k_{\overline{n}}=M(4\rho)-\eta_{\overline{n}}\omega(4\rho)$, $r=2\rho$ and $\ell=\frac12$, so that
\begin{equation*}
\frac{1}{\Lambda|B_{2\rho}(Z)|_{y^{1-2s}}}\int_{A_+^*(M(4\rho)-\eta_{\overline{n}}\omega(4\rho),2\rho)}\mkern-45mu y^{1-2s}|W-(M(4\rho)-\eta_{\overline{n}}\omega(4\rho))|^2dxdy\leq (\eta_{\overline{n}}\omega(4\rho))^2=d^2,
\end{equation*}
we obtain,
\begin{equation*}
W(X)\leq k+\ell d\leq [M(4\rho)-\eta_{\overline{n}}\omega(4\rho)]+\frac{1}{2}\eta_{\overline{n}}\omega(4\rho)\leq M(4\rho)-\frac{1}{2}\eta_{\overline{n}}\omega(4\rho), \quad \mbox{ a.e. in } \ \mathscr{C}_{\Omega(z,R)(Z,\rho)}\,.
\end{equation*}
As a consequence,
\begin{align*}
\omega(\rho)&=M(\rho)-m(\rho)\leq M(\rho)-m(4\rho)\leq [M(4\rho)-\frac{1}{2}\eta_{\overline{n}}\omega(4\rho)]-m(4\rho) \leq (1-\frac{1}{2}\eta_{\overline{n}})\omega(4\rho),
\end{align*}
and we deduce \eqref{ineq:oscillation} by choosing  $\overline{\eta}=(1-\eta_{\overline{n}+1})$.
\end{proof}

The next result gives an estimate on the growth of the oscillation.

\begin{theorem}\label{pre2_holder}
Given $z_0\in\overline{\Omega}$ and $R>0$,    let  $W\in \mathcal{X}_{\Sigma_{\mathcal{D},R}}^s(\mathscr{C}_{\Omega(z_0,R)})$ be a solution to the homogeneous problem \eqref{homo}. Then, there exist $0<\mathcal{H}<1$ and $0<\tau<\frac{1}{2}$  such that for any $Z\in\mathscr{C}_{\Omega(z_0,R)}^{\circ}$ there exists $\delta(Z)>0$ such that
\begin{equation*}
\omega(\rho)=\sup\limits_{X\in\overline{\mathscr{C}}_{\Omega(z_0,R)}(Z,\rho)}\mkern-10mu W(X)-\inf\limits_{X\in \overline{\mathscr{C}}_{\Omega(z_0,R)}(Z,\rho)}\mkern-10mu W(X)\leq\mathcal{H}\rho^{\tau},
\end{equation*}
for any  $0<\rho<\delta(Z)$.
\end{theorem}

\begin{proof}
Let $r(Z)=\min\{\widetilde{\rho}(Z),\overline{\rho}(Z)\}$,  by  Theorem \ref{th:pre_holder}, inequality \eqref{ineq:oscillation}
holds true for any $\rho<r(Z)/4$. Take $\tau$, $M$ positive such that $4^{\tau}\overline{\eta}=a<1$ and $\omega(\rho)\leq M\rho^{\tau}$ for $\frac{r(Z)}{4}\leq\rho< r(Z)$.
Then, again by  \eqref{ineq:oscillation}, we have that
\begin{equation*}
\omega(\rho)\leq\overline{\eta}4^{\tau}M\rho^{\tau},
\end{equation*}
for $\frac{r(Z)}{4^{2}}\leq\rho<\frac{r(Z)}{4}$.
In general, if $\frac{r(Z)}{4^{i+1}}\leq\rho<\frac{r(Z)}{4^{i}}$ for some $i\in \mathbb{N}$, we deduce that
$\omega(\rho)\leq \left(\overline{\eta}4^{\tau}\right)^iM\rho^{\tau}$.
Letting $\overline{i}$ large enough such that $\mathcal{H}=Ma^{\overline{i}}<1$, we obtain $\omega(\rho)\leq\mathcal{H}\rho^{\tau}$ for any $\rho<\delta(Z)=\frac{r(Z)}{4^{\overline{i}}}$. On the other hand, since we have chosen $\tau>0$ such that $4^{\tau}\overline{\eta}<1$ and, by  Theorem \ref{th:pre_holder},  $\overline{\eta}=1-\eta_{\overline{n}+1}$ for some $\overline{n}\geq0$ independent from $Z$ and $\rho$, it follows that
\begin{equation}\label{exponente_regularidad}
\tau<\frac{1}{2}\log_2\left(\frac{2^{\overline{n}+2}}{2^{\overline{n}+2}-1}\right)<\frac12.
\end{equation}
\end{proof}

Before  proving Theorem \ref{holder_result}, let us observe the following:
\begin{itemize}
\item[(i)] if $z_0 \in\Omega$, then there exist $R>0$ sufficiently small such that $\Sigma_{\mathcal{D},R}=\Sigma_{\mathcal{N},R}=\emptyset$  and $\widetilde{\rho}(Z)=dist(Z,\partial_L\mathscr{C}_{\Omega(z_0 ,R)})$ for any $z  \in\mathscr{C}_{\Omega(z_0 ,R)}$.
\item[(ii)] if $z_0 \in \Sigma_{\mathcal{D}}\backslash\Gamma$, then there exist $R>0$ such that $\Sigma_{\mathcal{N},R}=\emptyset$. Hence $\widetilde{\rho}(Z)=dist(Z,\partial_B\mathscr{C}_{\Omega(z_0 ,R)})$ for any $Z\in\Sigma_{\mathcal{D},R}^*$ and $\widetilde{\rho}(Z)=dist(Z,\partial_0\mathscr{C}_{\Omega(z_0 ,R)})$ for any $Z\in\mathscr{C}_{\Omega(z_0 ,R)}^{\circ}\backslash\Sigma_{\mathcal{D},R}^*$.
\item[(iii)] if $z_0 \in \Sigma_{\mathcal{N}}$, then there exist $R>0$ such that $\Sigma_{\mathcal{D},R}=\emptyset$. Hence we have $\widetilde{\rho}(Z)=dist(Z,\partial_B\mathscr{C}_{\Omega(z_0 ,R)})$ for any $Z\in\mathscr{C}_{\Omega(z_0 ,R)}^{\circ}$.
\item[(iv)] if $z_0 \in\Gamma$ then $\forall R>0 \ $  both $\Sigma_{\mathcal{D},R}\neq\emptyset$ and $ \Sigma_{\mathcal{N},R}\neq\emptyset$
and hence $\widetilde{\rho}(Z)=dist(Z,\partial_B\mathscr{C}_{\Omega(z_0 ,R)})$ for any $Z\in\Sigma_{\mathcal{D},R}^*$ and $\widetilde{\rho}(Z)=dist(Z,\partial_0\mathscr{C}_{\Omega(z_0 ,R)})$ for any $Z\in\mathscr{C}_{\Omega(z_0 ,R)}^{\circ}\backslash\Sigma_{\mathcal{D},R}^*$.
\end{itemize}

Now, consider $\overline{\mathscr{C}}_{\Omega(z_0,R/2)}\subset \mathscr{C}_{\Omega(z,R)}$ if $z\in\Omega$ and $\overline{\mathscr{C}}_{\Omega(z_0,R/2)}\subset\mathscr{C}_{\Omega(z,R)}^{\circ}$ if $z\in\partial\Omega$.

Thus we deduce that:
\begin{itemize}
\item[(i)] if $z\in\Omega$, then $\widetilde{\rho}(Z)=dist(Z,\partial_L\mathscr{C}_{\Omega(z,R)})\geq\widetilde{\rho}>0$ for any $Z\in\overline{\mathscr{C}}_{\Omega(z,R/2)}$ and some positive $\widetilde{\rho}$.

\item[(ii)] if $z\in\Sigma_{\mathcal{D}}\backslash\Gamma$, then $\widetilde{\rho}(Z)=\widetilde{\rho}>0$ for some positive $\widetilde{\rho}$ for any $Z\in\Sigma_{\mathcal{D},R/2}^*$ and $\widetilde{\rho}(Z)=dist(Z,\Sigma^*_{\mathcal{D},R/2})$ for any $Z\in\overline{\mathscr{C}}_{\Omega(z,R/2)}\backslash\Sigma_{\mathcal{D},R/2}^*$.

\item[(iii)] if $z\in \Sigma_{\mathcal{N}}$, then $\widetilde{\rho}(Z)=dist(Z,\partial_B\mathscr{C}_{\Omega(z,R)})\geq\widetilde{\rho}>0$ for any $Z\in\overline{\mathscr{C}}_{\Omega(z,R/2)}$ and some positive $\widetilde{\rho}$.

\item[(iv)] if $z\in\Gamma$ then $\widetilde{\rho}(Z)=\widetilde{\rho}>0$ for some positive $\widetilde{\rho}$ for any $Z\in\Sigma_{\mathcal{D},R/2}^*$ and $\widetilde{\rho}(Z)=dist(Z,\Sigma_{\mathcal{D},R/2})$ for any $Z\in\overline{\mathscr{C}}_{\Omega(z,R/2)}\backslash\Sigma_{\mathcal{D},R/2}^*$.
\end{itemize}
Observe that if  either (i) or (iii) holds true  then the number
$0<\delta(Z)$ in Theorem \ref{pre2_holder} has an infimum value, namely $0<\delta<\delta(Z)$ for any $Z\in\overline{\mathscr{C}}_{\Omega(z_0,R/2)}$
and we deduce that solutions $W$ to problem \eqref{homo} are H\"older continuous up to the boundary of $\mathscr{C}_{\Omega(z_0,R/2)}$.
In fact, let us consider two points $Z_1$ and $Z_2$ in $\overline{\mathscr{C}_{\Omega(z_0,R)}^m}$  with $m>0$. Then, by  Corollary
\ref{cor:infty_boundhomo} and Theorem \ref{pre2_holder} we find
\begin{itemize}
\item If $|Z_1-Z_2|\geq\delta$, we have
$$\frac{|W(Z_1)-W(Z_2)|}{|Z_1-Z_2|^{\tau}}\leq\frac{2}{\delta^{\tau}}\max\limits_{\mathscr{C}_{\Omega(z_0,R/2)}^m}W=\frac{2}{\delta^{\tau}}\|W\|_{L^{\infty}(\mathscr{C}_{\Omega(z_0,R/2)}^m)}.$$
\item If $|Z_1-Z_2|<\delta$,  by Theorem \ref{pre2_holder},  $\frac{|W(Z_1)-W(Z_2)|}{|Z_1-Z_2|^{\tau}}\leq\mathcal{H}$, $0<\mathcal{H}<1$.
\end{itemize}
We conclude the H\"older regularity with a constant
\begin{equation}\label{eq:h-constant}
\mathcal{T}=\max\{\mathcal{H},\frac{2}{\delta^{\tau}}\|W\|_{L^{\infty}(\mathscr{C}_{\Omega(z,R/2)}^m)}\}.
\end{equation}
Now we deal with the situation described in items (ii) and (iv).

\begin{theorem}\label{holder}
For any  $z_0\in\Sigma_{\mathcal{D}}$ and $R>0$ let  $W\in \mathcal{X}_{\Sigma_{\mathcal{D},R}}^s(\mathscr{C}_{\Omega(z_0,R)})$ be  a solution to the homogeneous problem \eqref{homo}.
Then   $W\in\mathcal{C}_{loc}^{\tau}(\overline{\mathscr{C}}_{\Omega(z_0,R/2)})$ for some $0<\tau<\frac12$.
\end{theorem}

\begin{proof}
Observe that  the number $0<\delta(Z)$ in Theorem \ref{pre2_holder} is bounded
from below by some $0<\delta_{\texttt{H}}$ for $Z\in\Sigma_{\mathcal{D},R/2}^*$ and we can assume that $\delta(Z)\geq\min
\left\{\delta_{\texttt{H}},dist(Z,\Sigma_{\mathcal{D},R/2}^*)\right\}$ for $Z\in\Sigma_{\mathcal{N},R/2}^*$. Moreover, by
the construction of the lateral boundary of the extension cylinder, the numbers $\delta(Z)$ do not depend on the $y$ variable.
Hence such an  infimum   $\delta_{\texttt{H}}>0$ is attained at those points  of the type $Z=(z,0)$ in $\partial\Omega\times \{0\}$. Consider the set
\begin{equation*}
\mathscr{C}_{\Omega(z_0,R/2)}^{\delta}=\{Z\in\overline{\mathscr{C}^m_{\Omega(z,R/2)}}:\ dist(Z,\Sigma_{\mathcal{D},R/2}^*)\geq\delta_{\texttt{H}}\}.
\end{equation*}
As above, we only need to study the case $|Z_1-Z_2|<\delta_{\texttt{H}}$. Suppose that $Z_1\in\mathscr{C}_{\Omega(z_0,R/2)}^{\delta}$,
then $|Z_1-Z_2|\leq\delta_{\texttt{H}}<dist(Z_1,\Sigma_{\mathcal{D},R/2}^*)=\delta(Z_1)$, and thus, by  Theorem \ref{pre2_holder}, we have
\begin{equation*}
\frac{|W(Z_1)-W(Z_2)|}{|Z_1-Z_2|^{\tau}}\leq \mathcal{H}.
\end{equation*}
If neither $Z_1$ nor $Z_2$ belongs to $\mathscr{C}_{\Omega(z_0,R/2)}^{\delta}$ but one of them, say $Z_1\in\Sigma_{\mathcal{D},R/2}^*$,
we have $|Z_1-Z_2|\leq\delta_{\texttt{H}}=\delta(Z_1)$, and the results follows as before. If, instead, none of them belongs neither to
$\mathscr{C}_{\Omega(z_0,R/2)}^{\delta}$ nor to $\Sigma_{\mathcal{D},R/2}^*$, we have two cases:
\begin{itemize}
\item $|Z_1-Z_2|\leq\max\left\{dist(Z_1,\Sigma_{\mathcal{D},R/2}^*),dist(Z_2,\Sigma_{\mathcal{D},R/2}^*)\right\}$.
\item $|Z_1-Z_2|>\max\left\{dist(Z_1,\Sigma_{\mathcal{D},R/2}^*),dist(Z_2,\Sigma_{\mathcal{D},R/2}^*)\right\}$.
\end{itemize}
In the first case at least one of the two points, say $Z_1$, satisfies the inequality $|Z_1-Z_2|\leq\delta_{\texttt{H}}
<dist(Z_1,\Sigma_{\mathcal{D},R/2}^*)=\delta(Z_1)$ and we have the result as before. In the second case, there exists at least one
$\overline{Z}\in\Sigma_{\mathcal{D},R/2}^*$ such that $|\overline{Z}-Z_1|\leq|Z_1-Z_2|$, and using the triangle inequality it follows that
$|\overline{Z}-Z_2|\leq2|Z_1-Z_2|$. Since the result has been proved for the case when at least one point belongs to $\Sigma_{\mathcal{D},R/2}^*$,
we find
\begin{equation} \label{eq:paraluego}
|W(Z_1)-W(Z_2)|\leq|W(Z_1)-W(\overline{Z})|+|W(\overline{Z})-W(Z_2)|\leq 3\mathcal{H}|Z_1-Z_2|^{\tau},
\end{equation}
and we conclude the H\"older regularity with constant
$\mathcal{T}=\max\{3\mathcal{H},2\delta_{\texttt{H}}^{-\tau}\|W\|_{L^{\infty}(\mathscr{C}_{\Omega(z,R/2)})}\}$, with
$0<\mathcal{H}<1$ given by Theorem \ref{pre2_holder}, see \eqref{eq:h-constant}.
\end{proof}
\begin{corollary}\label{holder_abajo}
Let $\Omega$ be a smooth domain such that $\Sigma_{\mathcal{D}}$, $\Sigma_{\mathcal{N}}$ satisfy hypotheses $(\mathfrak{B})$ and let $w$ be the
solution to problem \eqref{homo_abajo} with $z\in\overline{\Omega}$ and $R>0$. Then, the function $w\in\mathcal{C}^{\tau}(\overline{\Omega(z,R/2)})$ for some $0<\tau<\frac12$.
\end{corollary}
\begin{proof} 
Since $\Omega$   satisfies hypotheses $(\mathfrak{B})$,   there exists $0<\delta_{\texttt{H}}<\delta(Z)$ for $Z\in\Sigma_{\mathcal{D},R/2}^*$ and we can assume that $\delta(Z)\geq\min\!\left\{\delta_{\texttt{H}},dist(Z,\Sigma_{\mathcal{D},R/2}^*)\right\}$ for $Z\in\Sigma_{\mathcal{N},R/2}^*$, with $\delta(Z)$ given in Theorem \ref{pre2_holder}.\newline
Suppose that $z_1,z_2\in (\overline{\Omega(z,R/2)}) $:
\begin{itemize}
\item If $|z_1-z_2|\geq\delta_{\texttt{H}}$. Then, due to Corollary \ref{cor:infty_boundhomo} we have $\|w\|_{L^{\infty}(\Omega(z,R/2))}<\infty$ and, therefore,
$$\frac{|w(z_1)-w(z_2)|}{|z_1-z_2|^{\tau}}\leq\frac{2}{\delta_{\texttt{H}}^{\tau}}\max\limits_{\Omega(z,R/2)}w.$$
\item While for $|z_1-z_2|<\delta_{\texttt{H}}$, let us set $Z_1=(z_1,0)$ and $Z_2=(z_2,0)$,
$Z_1,Z_2\in\overline{\mathscr{C}}_{\Omega(z,R/2)}$, such that $|Z_1-Z_2|<\delta_{\texttt{H}}$. Then, as in \eqref{eq:paraluego} in Theorem \ref{holder},
$$\frac{|w(z_1)-w(z_2)|}{|z_1-z_2|^{\tau}}=\frac{|W(Z_1)-W(Z_2)|}{|Z_1-Z_2|^{\tau}}\leq3\mathcal{H},\quad 0<\mathcal{H}<1.$$
\end{itemize}
Hence, we conclude
\begin{equation*}
|w(z_1)-w(z_2)|\leq \mathcal{T}|z_1-z_2|^\tau,\quad \forall z_1,z_2\in\overline{\Omega(z,R/2)},
\end{equation*}
with $\mathcal{T}=\max\{3\mathcal{H},2\delta_{\texttt{H}}^{-\tau}\|w\|_{L^{\infty}(\Omega(z,R/2))}\}$, and $\delta_{\texttt{H}}>0$ given as above.
\end{proof}
We prove now the main result of this work.
\begin{proof}[Proof of Theorem \ref{holder_result}]
Let $u$ be the solution to problem \eqref{problema1}, $\Omega$ a smooth bounded domain such that $\Sigma_{\mathcal{D}}$, $\Sigma_{\mathcal{N}}$ satisfy
hypotheses $(\mathfrak{B})$ and $f\in L^{p}(\Omega)$ for $p>\frac{N}{2s}$. Given $z\in\overline{\Omega}$ and $0<R<1$, let $v$ be the solution
to \eqref{nonhomo_abajo} and $w=u-v$ a function satisfying \eqref{homo_abajo}. Thus, using \eqref{bound_nonhomo_aux} and Corollary\ref{holder_abajo},
we conclude that, for any $x,y\in\overline{\Omega(z,R/2)}$,
\begin{equation*}
\omega(u,R/2)\leq \omega(w,R/2)+2 \max_{x\in\Omega(z,R/2)}v(x)\leq \mathcal{T} R^{\tau}+C(N,s,|\Sigma_{\mathcal{D}}|)\|f\|_{L^p (\Omega (z,R))}  R^{2s-\frac{N}{p}}\leq
\mathcal{C}R^{\gamma},
\end{equation*}
where $\gamma=\min\{\tau,2s-\frac{N}{p}\}<\frac12$ and $\mathcal{C}=\max\{\mathcal{T},2C(N,s,|\Sigma_{\mathcal{D}}|)\|f\|_{L^p (\Omega(z,R) )} \}$, with
\begin{equation*}
\mathcal{T}=\max\{3\mathcal{H},2\delta_{\texttt{H}}^{-\tau}\|w\|_{L^{\infty}(\Omega(z,R/2))}\}=\max\{3\mathcal{H},2\delta_{\texttt{H}}^{-\tau}
\|u-v\|_{L^{\infty}(\Omega(z,R/2))}\}.
\end{equation*}
Moreover, by  Theorem \ref{th:bound_nonhomo},
$\|u-v\|_{L^{\infty}(\Omega(z,R/2))}\leq\|u\|_{L^{\infty}(\Omega(z,R))}+\|v\|_{L^{\infty}(\Omega(z,R))}\le 2 C(N,s,|\Sigma_{\mathcal{D}}|)\|f\|_{L^p (\Omega(z,R) )}$
hence we obtain
\begin{equation*}
\mathcal{T}\le \max\{3\mathcal{H},
4\delta_{\texttt{H}}^{-\tau}C(N,s,|\Sigma_{\mathcal{D}}|)\|f\|_p\}.
\end{equation*}
Therefore, $\mathcal{C}=\max\{3\mathcal{H},4\delta_{\texttt{H}}^{-\tau}C(N,s,|\Sigma_{\mathcal{D}}|)\|f\|_{L^p (\Omega(z,R) )}\}$.
Repeating the steps above in Theorem \ref{holder}, we conclude
\begin{equation}\label{proofholder}
|u(x)-u(y)|\leq\mathscr{H}|x-y|^{\gamma},\ \mbox{for any } x,y\in\overline{\Omega(z,R/2)},
\end{equation}
where
\begin{equation*}
\mathscr{H}=\max\left\{9\mathcal{H},\frac{C(N,s,|\Sigma_{\mathcal{D}}|)\|f\|_{L^p (\Omega(z,R) )}}{\delta_{\texttt{H}}^{\gamma}}\right\},
\end{equation*}
and $\gamma=\min\{\tau,2s-\frac{N}{p}\}<\frac12$. Since the constants $\mathscr{H}$ and $\gamma$ do not depend neither on $z$ nor on $R$, to complete the proof, set $z_i\in\overline{\Omega}$,
$i=1,2,\ldots,m$ and $R_i>0$, small enough such that
\begin{equation*}
\overline{\Omega}=\bigcup_{i=1}^{m}\Omega(z_i,R_i/4).
\end{equation*}
Then \eqref{proofholder} follows by using a suitable recovering argument. 
%Then, given $x,y\in\overline{\Omega}$, we can assume that $x,y\in\Omega(z_i,R_i/2)$ for some $i\geq1$ and, hence, we deduce
%\eqref{proofholder} in $\overline \Omega$.
\end{proof}

\section{Moving the boundary conditions}

In this last part, we study the behavior of the solutions to problem \eqref{problema1} when we move the boundary conditions.
First, let us describe this mixed moving boundary data framework.
As introduced above, given $I_{\varepsilon}=[\varepsilon,|\partial\Omega|]$, let us consider the family of closed sets
$\{\Sigma_{\mathcal{D}}(\alpha)\}_{\alpha\in I_{\varepsilon}}$, satisfying
\begin{itemize}
\item[$(B_1)$] $\Sigma_{\mathcal{D}}(\alpha)$ has a finite number of connected components.
\item[$(B_2)$] $\Sigma_{\mathcal{D}}(\alpha_1)\subset\Sigma_{\mathcal{D}}(\alpha_2)$ if $\alpha_1<\alpha_2$.
\item[$(B_3)$] $|\Sigma_{\mathcal{D}}(\alpha_1)|=\alpha_1\in I_{\varepsilon}$.
\end{itemize}
We call $\Sigma_{\mathcal{N}}(\alpha)=\partial\Omega\backslash\Sigma_{\mathcal{D}}(\alpha)$ and $\Gamma(\alpha)=
\Sigma_{\mathcal{D}}(\alpha)\cap\overline{\Sigma}_{\mathcal{N}}(\alpha)$. Observe that, under the hypotheses $(B_1)$--$(B_3)$,
the limit sets $\Sigma_{\mathcal{D}}(\alpha),\, \Sigma_{\mathcal{N}}(\alpha)$ as $\alpha\to\varepsilon^+$ are not degenerated sets
(for instance a Cantor-like set).\newline
For a family of this type we consider the corresponding family of mixed boundary value problems
\begin{equation} \label{pal}
        \left\{
        \begin{tabular}{rcl}
        $(-\Delta)^su=f$ & &in $\Omega$, \\
        $B_{\alpha}(u)\!=0$  & &on $\partial\Omega$, \\
        \end{tabular}
        \right.
        \tag{$P_{\alpha}^s$}
\end{equation}
where $B_{\alpha}(u)$ means $B(u)$  with $\Sigma_{\mathcal{D}}$, $\Sigma_{\mathcal{N}}$, and $\Gamma$ are replaced by $\Sigma_{\mathcal{D}}(\alpha)$,
$\Sigma_{\mathcal{N}}(\alpha)$, and $\Gamma(\alpha)$ respectively. Similarly, $(\mathfrak{B}_{\alpha})$ means $(\mathfrak{B})$ with the natural changes as above.

\medskip 

Our main aim here is to prove Theorem \ref{cor:reg}. 

\medskip

The key point in order to obtain it,  is to prove that we can choose $\beta_s>0$ in \eqref{sobA} independent of the measure of the Dirichlet part.
Nevertheless, as we will see below, when one takes $\alpha\to0^+$ the control of the H\"older norm of such a family is lost. Hence, it is necessary to fix a positive minimum $\varepsilon>0$ on the measure of the family $\{\Sigma_{\mathcal{D}}(\alpha)\}_{\alpha\in I_{\varepsilon}}$, in order to guarantee the control on the H\"older norm for the family $\{u_{\alpha}\}_{\alpha\in I_{\varepsilon}}$.
\begin{proof}[Proof of Theorem \ref{cor:reg}]
 Assume that $\partial\Omega$ is a smooth manifold and $\Sigma_{\mathcal{D}}(\alpha)$, $\Sigma_{\mathcal{N}}(\alpha)$ satisfy hypotheses $(\mathfrak{B})$. 
 %, i.e. $\mathscr{C}_{\Omega}$ is an uniform $\lambda$-admissible set for any $0<\lambda<1$. 
 Thus, there exists $\delta>0$ such that $\overline{\rho}(Z)\geq\delta$ for all $Z\in\partial_{L}\mathscr{C}_{\Omega}$. Then:
 %taking in mind (i)-(ii) in Remark \ref{rem:ladm}, we have the following.
\begin{enumerate}
\item If $Z\in\overline{\mathscr{C}}_{\Omega}\backslash\Sigma_{\mathcal{D}}^*(\alpha)$, inequality \eqref{sobA} holds true with $\beta_s=\frac{c_s}{\zeta\lambda}$ independent of $\alpha$, for all $0<\rho<\delta$.
\item If $Z\in\Sigma_{\mathcal{D}}^*(\alpha)\setminus \Gamma^*(\alpha)$, we can set $0<\rho<\min\{\delta,dist(Z,\Gamma^*(\alpha))\}$, such that for all $X\in\mathscr{C}_{\Omega}(Z,\rho)$,
\begin{equation*}
\Pi(X,\Sigma_{\mathcal{D}}^*\cap B_{\rho}(Z),\mathscr{C}_{\Omega}(Z,\rho))\geq\varphi>0,
\end{equation*}
with $\varphi$ independent from $\alpha$, 
recalling that (according to \cite[\S 4]{S}) 
$$%\begin{equation}\label{angleprojection}
\Pi(x_0,E,A)=|\mathcal{V}_{x_0}(E)\cap \mathbb{S}_{N-1}(x_0)|=|\mathscr{S}_{x_0}|.
$$%\end{equation}
with  $\mathcal{V}$ defined as follows:  
given $x_0\in A$ and a closed set $E\subset A$, let us consider the cone $\mathcal{V}_{x_0}(E)\subset A$ consisting on all rays starting at $x_0$ and ending at some point $P\in E$.

\smallskip

Hence, inequality \eqref{sobA} holds true with $\beta_s\leq\frac{c_s}{\varphi}$ also independent from  $\alpha$.
\item If $Z\in\Gamma^*(\alpha)$, we can assume without loss of generality that, for some neighborhood of radius
$0<\rho<\min\{\delta,\delta_{\Gamma}\}$ of the point $Z=(Z_1,\ldots,Z_{N+1})$, $\partial_L\mathscr{C}_{\Omega}$
coincides with the hyperplane $\mathbb{R}^{N+1}\cap\{x_{N}=0\}$ and $\Gamma^*(\alpha)\subset\mathbb{R}_+^{N+1}\cap\{x_{N}=0,x_{N-1}=0\}$,
in such a way that in $\Sigma_{\mathcal{D}}^*(\alpha)$ we have $x_{N-1}\geq0$ and, in $\Sigma_{\mathcal{N}}^*(\alpha)$ we have $x_{N-1}<0$.
Now, $\mathscr{C}_{\Omega}(Z,\rho)$ is transformed by the bi-Lipschitz transform (that in fact keeps the extension variable unchanged)
%\footnote{We only perform a transformation in the $x_1,\ldots,x_N$ variables, without change the extension variable $y$.},
\begin{equation*}
x_i=\xi_i,\quad i=1,2,\ldots,N-1,
\end{equation*}
\begin{equation*}
x_N=\left\{
        \begin{tabular}{ll}
        $\xi_N$ & if $\xi_{N-1}<0$, \\
        $\xi_N-\xi_{N-1}$ & if $\xi_{N-1}\geq0$, \\
        \end{tabular}
        \right.
\end{equation*}
into a set $\mathcal{O}_{\rho}(Z)=\mathcal{O}_{\rho}^1(Z)\cup\mathcal{O}_{\rho}^2(Z)$ with
\begin{align*}
\mathcal{O}_{\rho}^1(Z)=&\left\{\xi_{N}\geq0,\xi_{N-1}<0,\sum_{i=1}^{N}(\xi_{i}-Z_i)^2+(y-Z_{N+1})^2\leq \rho^2\right\},\\
\mathcal{O}_{\rho}^2(Z)=&
\left\{
\begin{tabular}{l}
$\displaystyle\xi_{N-1}\geq0, \sum_{i=1}^{N-1}(\xi_{i}-Z_i)^2+(y-Z_{N+1})^{2}\leq \rho^2,$\\
$\displaystyle \xi_{N-1}\leq\xi_{N}\leq\xi_{N-1}+\left(\rho^2-\sum_{i=1}^{N-1}(\xi_{i}-Z_i)^2-(y-Z_{N+1})^2\right)^{\frac12}$
\end{tabular}
\right\}.
\end{align*}
Moreover $\Sigma_{\mathcal{D}}^*\cap B_{\rho}(Z)$ is transformed into the set
\begin{equation*}
\mathcal{D}_{\rho}(Z)=\left\{\xi_{N}=\xi_{N-1},\xi_{N-1}\geq0, \sum_{i=1}^{N-1}(\xi_{i}-Z_i)^2 + (y-Z_{N+1})^2\leq\rho^2\right\}.
\end{equation*}
Given $X_0\in\mathcal{O}_{\rho}(Z)$, we  use again  the representation (see \cite[cfr. 13.1]{S}):
\begin{equation*}
\Pi(X_0,\mathcal{D}_{\rho}(Z),\mathcal{O}_{\rho}(Z))=\frac{1}{|\mathbb{S}_{N}(X_0)|}\int_{\mathcal{D}_{\rho}(Z)}\frac{1}{|X_0-Y|^{N}}\cos(\psi)d\sigma,
\end{equation*}
where $\cos(\psi)=\langle \frac{X_0-Y}{|X_0-Y|}, \vec{v}\rangle$, with $\vec{v}$ the normal vector to $\{\xi_{N}=\xi_{N-1}\}\cap\mathbb{R}_+^{N+1}$. Since $\cos(\psi)$ vanish only when $X_0\in\mathcal{D}_{\rho}(Z)$ we conclude that $\Pi(X_0,\mathcal{D}_{\rho}(Z),\mathbb{R}_+^{N+1})\geq\varphi>0$ for all $X_0\in\mathcal{O}_{\rho}(Z)$ and some $\varphi>0$ independent of $\alpha$. On the other hand, it is immediate that $\varphi$ is independent of $\mathcal{\rho}$. Hence, inequality \eqref{sobA} holds true with $\beta_s\leq\frac{c_s}{\varphi}$ also independent of $\alpha$.
\end{enumerate}
Let us define
\begin{equation}\label{rho_moving}
\overline{\rho}_{\alpha}(Z):=\left\{
\begin{tabular}{ll}
$\min\{\delta,dist(Z,\Sigma_{\mathcal{D}}^*)\}$,& if $Z\in\overline{\mathscr{C}}_{\Omega}\backslash\Sigma_{\mathcal{D}}^*(\alpha)$,\\
$\min\{\delta,dist(Z,\Gamma^*)\}$,& if $Z\in\Sigma_{\mathcal{D}}^*(\alpha)\setminus\Gamma^*(\alpha)$,\\
$\min\{\delta,\delta_{\Gamma}\}$,& if $Z\in\Gamma^*(\alpha)$.
\end{tabular}
\right.
\end{equation}
As a consequence of (1)--(3) above, we deduce
\begin{enumerate}
\item[(i)] by  \eqref{lambda2}, the constant $\Lambda$ appearing in Theorem \ref{thcota} %, Theorem \ref{thcota_inferior}
and Theorem
\ref{th:pre_holder}, is independent of $\alpha$. Hence, inequality \eqref{ineq:bound_infty} does not depends on $\alpha$ and also  the number
$0<\mathcal{H}<1$ in Theorem \ref{pre2_holder} is independent from $\alpha$.
\item[(ii)] by  \eqref{n_indep}, the constant $\overline{\eta}$ in Theorem \ref{th:pre_holder} is independent from $\alpha$ and,
by  \eqref{exponente_regularidad}, also  that $0<\gamma<\frac{1}{2}$ is independent from $\alpha$.
\end{enumerate}
Then, given $u_{\alpha}$ a solution to problem \eqref{pal} with $\alpha\in I_{\varepsilon}$, by  Theorem \ref{holder_result}, we deduce
\begin{equation*}
\|u_{\alpha}\|_{\mathcal{C}^{\gamma} (\Omega)}\leq \mathscr{H}_{\alpha},
\end{equation*}
with $\gamma=\min\{\tau,2s-\frac{N}{p}\}<\frac12$ independent of $\alpha$ and $\mathscr{H}_{\alpha}=\max\{9\mathcal{H},\frac{C(N,s,\alpha)
\|f\|_p}{\delta_{\texttt{H},\alpha}^{\tau}}\}$ with the constants $0<\tau<\frac12$ and $\delta_{\texttt{H},\alpha}$ given as in Corollary \ref{holder_abajo}. Now, if we consider the family $\{u_{\alpha}\}_{\alpha\in I_{\varepsilon}}$, since
$\overline{\rho}_{\alpha_1}(Z)\leq\overline{\rho}_{\alpha_2}(Z)$ it is clear that $\delta_{\texttt{H},\alpha_1}\leq\delta_{\texttt{H},\alpha_2}$
and, therefore, $\mathscr{H}_{\alpha_1}\geq\mathscr{H}_{\alpha_2}$ for all $\alpha_1,\alpha_2\in[\varepsilon,|\partial\Omega|]$,
$\alpha_1\leq\alpha_2$. Therefore, we can take $0<\gamma<\frac12$ and $\mathscr{H}_{\varepsilon}=\max\{9\mathcal{H},\frac{C(N,s,\varepsilon)\|f
\|_p}{\delta_{\texttt{H},\varepsilon}^{\tau}}\}$ independent from $\alpha$ such that
\begin{equation*}
\|u_{\alpha}\|_{\mathcal{C}^{\gamma} (\Omega)}\leq \mathscr{H}_{\varepsilon},
\end{equation*}
To conclude, we observe that the condition $\alpha\in[\varepsilon,|\partial\Omega|]$ is necessary in order to control the H\"older norm
of the family $\{u_{\alpha}\}_{\alpha\in I_{\varepsilon}}$. If we let $\alpha=|\Sigma_{\mathcal{D}}(\alpha)|\to0^+$, then it is clear that $|\Sigma_{\mathcal{D}}^*(\alpha)\cap\overline{\mathscr{C}}_{\Omega}(Z,\rho)|\to0$
 for any $Z\in\overline{\mathscr{C}}_{\Omega}$ and $\rho>0$. Thus, if $\alpha\to0^+$, we conclude from \eqref{rho_moving} that
 $\overline{\rho}_{\alpha}(Z)\to0$ for any $Z\in\Sigma_{\mathcal{D}}^*$ and, hence, $\delta_{\texttt{H},\alpha}\to0$ while  $\mathscr{H}_{\alpha}\to+\infty$. 
 \end{proof}
\begin{remark}
Given an interphase point $Z\in\Gamma^*$, it is clear from \eqref{rho_moving}, that we can choose an uniform $\rho_{\varepsilon}>0$ in
the lines of \cite[Corollary 6.1]{ColP}. In fact, it is enough to choose $\delta_{\Gamma}$ in \eqref{rho_moving} in such a way that
$\Sigma_{\mathcal{D}}^*(\varepsilon)\cap\overline{\mathscr{C}}_{\Omega}(Z,\rho)$ is contained in some hyperplane (see (3) in the proof
of Theorem  \ref{cor:reg}). Clearly, this Dirichlet boundary part, say $\left(\{x_{N}=0,x_{N-1}\geq0\}\cap\mathbb{R}_+^{N+1}\right)
\cap B_{\rho_{\varepsilon}}(Z)$ converges to an empty set as $\rho_{\varepsilon}\to0$. 
% and the control on the H\"older norm is lost.
\end{remark}

\end{document}